\documentclass[a4paper, 12pt]{article}
\usepackage[margin=2.5cm, top=2.5cm, bottom=5cm, footskip=3cm]{geometry}

\usepackage{graphicx} 
\usepackage{amsmath}		
\usepackage{amsfonts}		
\usepackage{amsthm}
\usepackage[utf8]{inputenc}
\usepackage{color}

\usepackage{tikz}
\usepackage{tikz-cd} 
\usepackage{slashed}

\usetikzlibrary{calc}
\usetikzlibrary{arrows}  
\usetikzlibrary{decorations.markings}
\usetikzlibrary{shapes}

\usepackage{hyperref} 
\hypersetup{
	unicode=true,
	pdftoolbar=true,
	pdfmenubar=true,
	pdffitwindow=false,
	pdfstartview={FitH},
	pdftitle={My title},
	pdfauthor={Author},
	pdfsubject={Subject},
	pdfcreator={Creator},
	pdfproducer={Producer},
	pdfkeywords={keyword1} {key2} {key3},
	pdfnewwindow=true,
	colorlinks=true,   
	linkcolor={blue!80!black}, 
	citecolor={CrimsonRed}, 
	filecolor={blue!80!black},
	urlcolor={blue},      
	hypertexnames=false, %
} 

\definecolor{RoyalBlue}{RGB}{35,32, 57}
\definecolor{lightblue}{RGB}{132,131,143}
\definecolor{CrimsonRed}{RGB}{130,0,29}

\usepackage[figure]{hypcap}
\usepackage{amssymb}		
\usepackage{amstext}		
\usepackage{graphicx}		
\usepackage[utf8]{inputenc}	
\usepackage[T1]{fontenc}	
\usepackage{pdfpages}		
\usepackage{wrapfig}		
\usepackage{here}			
\usepackage{float}			
\usepackage{fancyhdr}
\usepackage{verbatim}		
\usepackage{wasysym}		
\usepackage{tocloft}		
\usepackage{siunitx}		
\usepackage[arrow, matrix, curve]{xy} 
\usepackage{wrapfig,tabularx,multirow,hyperref,icomma}
\usepackage{mathptmx}
\usepackage{lmodern}
\usepackage{courier}
\usepackage{paralist}
\usepackage{afterpage}
\usepackage{subcaption}
\usepackage[section]{placeins}
{\left.\begin{aligned}}
	{\end{aligned}\right\rbrace}
\usepackage{titletoc}
\usepackage{titlesec}
\graphicspath{{figuresArXiv/}} 
\usepackage{etoolbox}
\usepackage{relsize,exscale}
\usepackage[titletoc]{appendix}
\usepackage{natbib}

\newcommand\numberthis{\addtocounter{equation}{1}\tag{\theequation}}

\newcommand{\AS}[1]{#1}
\newcommand{\RCA}[1]{#1}
\newcommand{\VK}[1]{#1}

\renewcommand{\i}{i}
\newcommand{\ptl}{\partial}
\newcommand{\vph}{\varphi}

\newcommand{\be}{\begin{equation}}
	\newcommand{\ee}{\end{equation}}
\newcommand{\beq}{\begin{equation}}
	\newcommand{\eeq}{\end{equation}}

\newcommand{\Sp}{\mathcal{S}}

\newcommand{\CP}{\mathbb{C}\mathbb{P}}
\renewcommand{\S}{\mathcal{S}}

\newcommand{\tmmathbf}[1]{\ensuremath{\boldsymbol{#1}}}

\newcommand{\bz}{\boldsymbol{z}}

\renewcommand{\c}{\!\mathbb{c}}

\usepackage{cancel}

\usepackage{stmaryrd} 
\newcommand{\class}[1]{\left\llbracket #1 \right\rrbracket}

\usepackage[scr=boondoxo,bb=libus]{mathalpha}

\newtheorem{definition}{Definition}[section]

\newtheorem{proposition}[definition]{Proposition}

\newtheorem{lemma}[definition]{Lemma}

\theoremstyle{definition}
\newtheorem{remark}[definition]{Remark}

\theoremstyle{definition}

\theoremstyle{definition}

\newcommand{\z}{\boldsymbol{z}}
\newcommand{\x}{\boldsymbol{x}}

\newcommand{\uu}{\boldsymbol{u}}
\newcommand{\vv}{\boldsymbol{v}}

\newcommand{\xNP}{\x_{\textsc{np}}}
\newcommand{\xSP}{\x_{\textsc{sp}}}

\newcommand{\bdeta}{\boldsymbol{\eta}}

\usepackage[normalem]{ulem}
\usepackage{graphicx}
\renewcommand{\sout}[1]{}

\numberwithin{equation}{section}

\usepackage{authblk}

\title{Plane-wave representation for the Laplace--Beltrami equation on a sphere: \\ application to the Green's function.
}

\author[1]{Andrey V. Shanin}
\affil[1]{{\footnotesize Department of Physics (Acoustics Division), Moscow State University, Leninskie Gory,
119992 Moscow, Russia}}
\author[2]{Valentin D. Kunz}
\affil[2]{{\footnotesize Department of Mathematics, Ohio State University, 231 W 18th Ave, Columbus, OH 43210, USA}}
\author[3]{Rapha\"el C. Assier}
\affil[3]{{\footnotesize Department of Mathematics, The University of Manchester, 
Oxford Road, Manchester, M13 9PL, UK}}

\date{\today}

\begin{document}
	
\maketitle
\begin{abstract}
 We propose an extension of the plane-wave representation for wave fields defined on the real sphere $\S^2$. This representation is well-known in the planar setting but has never been developed for curved surfaces. To achieve this, we need to carefully study the geometry of the complexification of $\S^2$ and the properties of the Laplace--Beltrami operator, while using  concepts of multidimensional complex analysis. We extend the region of validity of such plane-wave representation by developing a sliding-contours method. Our methodology is illustrated through the study of the Green's function on the real sphere.
\end{abstract}

\section{Introduction}	
We consider the Laplace--Beltrami equation on the unit sphere $\mathcal{S}^2 \subset \mathbb{R}^3$. That is, an equation comprising the Laplace--Beltrami operator and  a multiplication by a constant. Such an equation is known to have wave-like solutions, and is a natural generalisation of the Helmholtz
equation on a plane.

A Laplace--Beltrami equation, besides being interesting in itself, can be important as a part of building a solution for problems of diffraction by cones  as initially explained in \citep{Smyshlyaev1993, Smyshlyaev2000} and further exploited in e.g.\  \citep{shanin2005modified, AssierPeake2012,shanin2012asymptotics, AssierPoonPeake2016,lyalinov2018acoustic,Sarigiannidis2025}.

Our aim is to develop a theoretical framework for the Laplace--Beltrami equation that would be close to that developed for the Helmholtz equation on a plane. Namely, such a framework for the Helmholtz equation uses the following concepts: 

\begin{itemize}

\item
{\em Plane waves} which are elementary ``building blocks'' for more complicated solutions. 

\item
Contour integral representations of the fields that have the structure of {\em plane-wave decompositions\/}
(see e.g.\ \citep{Babich2007}).

\item 
The {\em directivity pattern\/} of the field and its connection with the plane-wave decomposition.  

\item 
A {\em spectral ordinary differential equation\/} for the directivity pattern of the field
(see \citep{williams1982diffraction,Shanin2001,Shanin2003}). 
\end{itemize}

In the present work we address the first two points from this list; we plan to focus on the two remaining points, and to use the resulting framework for the study of  scattering processes on the sphere, in the near future. 

To achieve this, we consider the simplest non-trivial problem: the inhomogeneous Laplace--Beltrami equation on an entire sphere with a single 
point source forcing.  Our overall aim is to obtain a plane-wave representation formula for the resulting Green's function.

The structure of the paper is as follows. In Section~\ref{s:LaplaceBeltrami}, we formulate the Laplace--Beltrami equation on the real sphere $\S^2$ (of real dimension 2). We then 
define the complexification of this sphere, the complex sphere $\S^2_{\c}$ (of real dimension 4), and compactify it 
by adding to it an additional sphere (of real dimension 2) at infinity, which we call $\Xi$.  
The geometry of the complex sphere $\S^2_{\c}$ is studied by introducing two families of complex characteristic lines, and  
a complexified Laplace--Beltrami equation is introduced on $\S^2_{\c}$.

In Section~\ref{s:PlaneWaves}, we introduce the concept of plane waves on both $\S^2$ and $\S^2_{\c}$. They are defined by considering a Green's function
for which the point source is carried to infinity. These plane waves are shown to be elementary functions having branch sets within $\S^2_{\c}$. 

In Section~\ref{sec:Plane-Wave}, we obtain a plane-wave representation of the Green's function on the real sphere $\S^2$. This representation is 
first constructed in a ``safe zone'' using the standard integral formula for the Legendre functions. 
Then, this representation is continued to the whole sphere using the sliding contours concept. Finally, the 
representation is analytically continued to the complex sphere $\S^2_{\c}$ using contour integrals. 
The representation of the Green's function using the sliding contours is the main result of the current paper. 

In the appendices, we
prove the properties of the characteristic lines (Appendix~\ref{append_A}), \RCA{express the Laplace--Beltrami operator in different coordinate systems (Appendix~\ref{app:usefulformulas})}, motivate the introduction of the 
sliding contours (Appendix~\ref{app:planar}), and discuss some symmetries of the plane-wave representation for the Green's function (Appendix~\ref{app:sym}). 
\AS{
We end by explaining in elementary terms how to numerically evaluate integrals on a complex manifold in Appendix~\ref{app:int}.
}


\section{Laplace--Beltrami equation on a sphere}
\label{s:LaplaceBeltrami}

\subsection{Laplace--Beltrami equation on a real sphere. The Green's function}

Consider the real unit sphere $\S^2 \subset \mathbb{R}^3$: 
\begin{equation}
	\Sp^2
	= 
	\{ 
	\x  \in \mathbb{R}^3  \,\,  :  \,\, 
	x_1^2 + x_2^2 + x_3^2 = 1
	\},
	\label{e:02301}
\end{equation}
where we have introduced the notation $\x=(x_1,x_2,x_3)$. On $\S^2$, we introduce the standard coordinate system  $(\theta,\varphi)$  which, for $\theta \in [0,\pi]$, and $\varphi \in [0,2\pi)$, parametrises $\x$ as follows:
\begin{align}
x_1=\sin \theta \, \cos \varphi, 
\quad 
x_2=\sin \theta \, \sin \varphi , 
\quad 
x_3 = \cos \theta,
\label{Ch02.eq.01}
\end{align}
as shown in Figure \ref{f:02301}, left.

We will use both the Euclidean coordinate system in $\mathbb{R}^3$ and the $( \theta, \varphi )$ system to refer to the points of~$\S^2$.
We introduce the notations 
\[
\xNP = (0,0,1), \qquad \xSP = (0,0,-1),
\]
for the North pole and the South pole of $\S^2$, respectively. Indeed, these points correspond to $\theta = 0$ and $\theta= \pi$.

We consider a Helmholtz-type equation on $\S^2$  
with a point source forcing located at $\x= \x_0$. This equation, that we call the forced Laplace--Beltrami Equation, 
is given by 
\begin{equation}
	\left( \tilde \Delta + \lambda(\lambda+1) \right) G(\x;\x_0) = \delta(\x-\x_0),
	\label{e:02303}
\end{equation}
where $\lambda \in \mathbb{C}\setminus \mathbb{Z}$. The Laplace--Beltrami operator (LBO) $\tilde \Delta$ is defined as  
		 \begin{align}
		 	\tilde \Delta = {\rm div} \nabla, \label{eq.LB-CoordinateFree}
		 \end{align}
where ${\rm div}$ and $\nabla$ are the usual (real) divergence and gradient operators defined  on $\S^2$ relative to the metric that $\S^2$ inherits from the ambient Euclidean $\mathbb{R}^3$. As we will need it later on, we  also introduce the homogeneous Laplace--Beltrami equation 
\begin{equation}
\left( \tilde \Delta + \lambda(\lambda + 1) \right) u = 0.
\label{e:cdd002}
\end{equation}

Importantly, note that \eqref{eq.LB-CoordinateFree} defines $\tilde{\Delta}$ independently of any choice of coordinate system \cite{KobayashiNomizu}.
	
\begin{figure}[h]
	\centering
	\includegraphics[width=.75\textwidth]{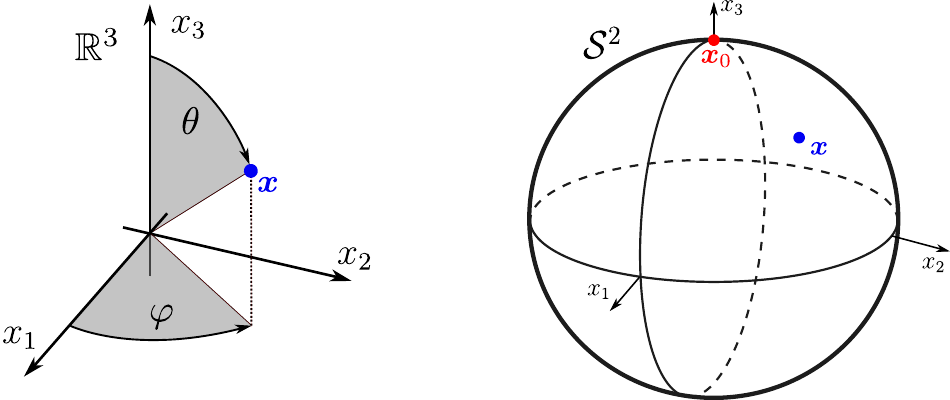}
	\caption{\small $(\theta,\varphi)$ coordinate system on $\S^2$ (left), and $\S^2 \subset \mathbb{R}^3$ with a point source located at $\x_0=\xNP$ (right).}
	\label{f:02301}
\end{figure}

In the coordinate system $(\theta, \varphi)$, the LBO $\tilde \Delta $ is written as
\begin{equation}
	\tilde \Delta  =
	\frac{\ptl^2}{\ptl \theta^2}
	+ 
	\frac{\cos \theta}{\sin \theta} \frac{\ptl}{\ptl \theta}
	+ 
	\frac{1}{\sin^2 \theta} \frac{\ptl^2 }{\ptl \vph^2} \cdot
	\label{e:02304}
\end{equation}
This expression makes sense everywhere except for $\theta = 0$ and $\theta=\pi$, but the LBO itself is well defined there and can be written explicitly by a simple rotation of the coordinate system. As it is known (see e.g.\ \cite{keller1999singularities}), the eigenvalues of the LBO $\tilde \Delta$ on $\S^2$ are of the form $-\lambda(\lambda+1)$ for $\lambda \in \mathbb{Z}$. Therefore, for $\lambda \notin \mathbb{Z}$, the operator $\tilde \Delta + \lambda(\lambda+1)$ is invertible, and
\eqref{e:02303} uniquely defines $G(\x;\x_0)$ (see e.g.\ \cite{Smyshlyaev1993}).


The solution to \eqref{e:02303} can be constructed relatively easily. Choosing $\x_0 = \xNP$ and
using the rotational invariance of \eqref{e:02303}, 
we find that $G(\x;\xNP)$ does not depend on~$\varphi$ and is therefore a function of $\theta$ only, that we  denote by $G(\theta)$ for simplicity. Thus, \eqref{e:02303} becomes 
rewritten as Legendre's differential equation 
\begin{equation}
	\left( 
	\frac{d^2}{d \theta^2} + \frac{\cos \theta}{\sin \theta}\frac{d }{d \theta} + \lambda (\lambda + 1)
	\right) G(\theta) =0 ,
	\qquad 
	0 < \theta < \pi.
	\label{e:02308}
\end{equation}

Similarly to the planar setting, it can be shown that the point-source forcing corresponds to the local singular behaviour
	\begin{equation}
	G(\theta) = \frac{1}{2\pi} \log(\theta)  + O(1) ~ \mbox{ as } ~ \theta \to 0,
	\label{e:02307}
	\end{equation}
 which supplements \eqref{e:02308}. 
Moreover, $G(\x; \xNP)$ should be regular at $\x = \xSP$ (i.e.\ $G(\theta)$ should be regular at $\theta=\pi$). 
Using this, it can be shown that the solution is given by 
\begin{equation}
			G(\x;\xNP) = \frac{1}{4 \sin (\pi \lambda)} P_{\lambda} (- \cos \theta),
			\label{e:02309}
		\end{equation}
where $P_\lambda $ denotes the \textit{Legendre function of the first kind} (see e.g.\ \cite{Smyshlyaev1993,Szmytkowski2007}), which is known to be regular at $\theta=\pi$ (see \cite{Bateman1953}).

Using the rotational invariance of $\tilde{\Delta}$,
we can generalise \eqref{e:02309} allowing $\x_0$ to be an arbitrary point of~$\S^2$:
\begin{equation}
	G(\x; \x_0) = \frac{1}{4\sin (\pi \lambda) } P_{\lambda} (- \x \cdot \x_0),
    \qquad  
    \x, \x_0 \in \S^2.
	\label{e:02310}
\end{equation}
Here, $\x \cdot \x_0$ denotes the dot-product of the vectors $\x = (x_1,x_2,x_3)$
and $\x_0 = (x_1',x_2',x_3')$ given by
\begin{equation}
\x \cdot \x_0 = x_1 x_1' + x_2 x_2' + x_3 x_3'.
\label{e:dot} 
\end{equation}

The aim of the rest of the paper is to build an analogue of a plane-wave representation for the Green's function~\eqref{e:02310}.

\begin{remark}
	The fact that \eqref{e:02303} uniquely determines the Green's function $G(\x;\x_0)$ is fundamentally different from what is required to uniquely determine Green's functions within the planar 2D-scattering setting \cite{Duffy2015} where one needs Sommerfeld's radiation condition to specify the Green's function's behaviour at infinity. 
\end{remark}

\begin{remark}
	The case of $\lambda \in \mathbb{Z}$ is {\em resonant}, and, as already mentioned, not considered throughout this article. Indeed, one cannot build a stationary Green's function for such~$\lambda$.
	Instead, there exist solutions of the homogeneous equation (\ref{e:02303}) on the whole sphere. Such solutions are expressed through Legendre's polynomials, and commonly referred to as spherical harmonics.
\end{remark}


\subsection{The complex sphere and its compactification}

We can define an affine complexification of the real sphere $\S^2$, denoted $\S^2_{\c}$, as follows: 
\begin{equation}
\S^2_{\c} = \{\z = (z_1,z_2, z_3) \in \mathbb{C}^3 \,\, : \, \,  z_1^2+z_2^2+z_3^2=1\} . 
\label{e:f001}
\end{equation}
The real sphere $\S^2$ is embedded into $\S^2_{\c}$ by taking $\z = \x$.
We will use the notation $\x$ (or $\z$) to indicate that a real (or complex) point is taken.

The complex sphere $ \S^2_{\c}$ can be described by the coordinates $( \theta , \varphi )$
(both taking complex values) almost everywhere via formulas similar to \eqref{Ch02.eq.01}: 
\begin{equation}
z_1 = \sin \theta \, \cos \varphi , 
\quad 
z_2 = \sin \theta \, \sin \varphi , 
\quad 
z_3 = \cos \theta.
\label{e:f003}
\end{equation}
This fact follows from the  Bruhat--Whitney theorem (see for instance  \cite[Chapter 5]{CieliebakEliashberg2010}).
The coordinates $(\theta, \vph)$ do not work regularly near the points~$\z = (0,0,\pm 1)$.
Neighbourhoods of these points
can be covered by, say, interchanging $z_2$ and $z_3$ in \eqref{e:f003}.

Let us now build a compactification of $\S^2_{\c}$. The most natural way to do this is to embed $\mathbb{C}^3$ into 
the projective complex space $\CP^3$ and close $\S^2_{\c}$ therein. \VK{In order to keep the present article self-contained,
we give the definition of $\CP^{n}$ below (cf. \cite{griffiths1984principles} or \cite{Shabat1991}, for instance) for any $n\in\mathbb{N}$, $n\geq1$. In the paper we will use mainly $n=2$ and $n=3$.}

\VK{Let $\uu, \vv \in \mathbb{C}^{n+1}\setminus \{\boldsymbol{0}\}$, with $\boldsymbol{0} = (0,..,0)$ denoting the zero-vector. Define the following equivalence relation $\sim$ by 
    \begin{equation}
        \uu \sim \vv, \qquad \text{iff} \qquad \uu = \lambda \vv, \quad \text{for} \quad \lambda \in \mathbb{C}\setminus \{0\}.
    \end{equation}
 The projective space $\CP^n$ is defined as the quotient space associated with this equivalence relation: 
    \begin{align}
        \CP^n = (\mathbb{C}^{n+1}\setminus \{\boldsymbol{0}\}) / \sim. 
    \end{align}
 The corresponding equivalence class of a vector $\uu=(u_0,u_1,\dots,u_n)\in\mathbb{C}^{n+1}$ will be denoted by ${\class{\uu}=\class{u_0:u_1:...:u_n}}$.  }

 \VK{$\CP^n$ is a complex manifold of dimension $n$ (see again \cite{griffiths1984principles,Shabat1991}), and one may define an atlas as follows. Let $U_j = \{\uu=(u_0,u_1,...,u_n) \in \mathbb{C}^{n+1}: ~ u_j \neq 0\}$ and let $V_j = \pi(U_j) \subset  \CP^n$ with $\pi(\uu) = \class{\uu}$ denoting the projection map $\pi: \mathbb{C}^{n+1} \to \CP^{n}$. Moreover, we let 
 \begin{align}
     \phi_j: V_j \to \mathbb{C}^n; \qquad \class{u_0:u_1:...:u_n} \mapsto \left(\frac{u_0}{u_j},..., \frac{u_{j-1}}{u_j},\frac{u_{j+1}}{u_j}, ... , \frac{u_n}{u_j} \right).
 \end{align}
 The collection of all $(V_j,\phi_j)$ forms an atlas of $\CP^n$. 
 Moreover, we will require the following \textit{embedding} of $\mathbb{C}^n$ within $\CP^n$. 
    \begin{align}
        \iota:  \mathbb{C}^n \to \CP^n; \quad \z=(z_1,...,z_n) \mapsto \class{1:z_1:...:z_n}.
    \end{align}
 Note that the embedding $\iota$ is the inverse mapping of $\phi_0$, and the inverse mappings of $\phi_j$ for $j\neq 0$ can be constructed similarly.}

 \VK{
 Finally, let us introduce the following notation. For $\mathcal{U} \subset \CP^n$, we let
    \begin{align}
        f_0(\mathcal{U}) = \phi_0(\mathcal{U} \cap V_0), \label{e:revVK2.17}
    \end{align}
and call $f_0(\mathcal{U})$ the \textit{finite part} of $\mathcal{U}$. 
}

\VK{This is all the projective-space machinery required for the present article. Using these concepts, we can now define the \textit{compactification} of $\mathcal{S}^2_{\c}$ as follows. 
    \begin{align}
        \hat{\mathcal{S}}^2_{\c} = \{\class{\uu}\in \CP^3: u_1^2 + u_2^2 + u_3^2 = u_0^2\}. \label{e:revVK2.18}
    \end{align}}
\VK{\begin{remark}
    $\hat{\mathcal{S}}^2_{\c}$ is actually the (unique up to isomorphism) quadric  surface within $\CP^3$, see \cite[Chapter 4]{griffiths1984principles}. 
\end{remark}}

\VK{Note that $\iota(\mathcal{S}^2_{\c}) \subset \hat{\mathcal{S}}^2_{\c}$ so, in this sense, we shall think of $\mathcal{S}^2_{\c}$ as being embedded within $ \hat{\mathcal{S}}^2_{\c}$. Indeed $f_0(\hat{\mathcal{S}}^2_{\c}) = \mathcal{S}^2_{\c}$, so the complexification of $\mathcal{S}^2$ is the finite part of $\hat{\mathcal{S}}^2_{\c}$.} 

\VK{Let us now discuss the infinity within $\hat{\mathcal{S}}^2_{\c}$. Within $\CP^3$, consider the so-called ``hyperplane at infinity''
    \begin{align}
        \CP^3_{\infty} = \{\class{\uu} \in \CP^3: u_0 = 0 \}. \label{e:revVK2.19}
    \end{align}
By definition of the embedding $\iota$, we have $\iota(\mathbb{C}^3) \cap \CP^3_{\infty} = \emptyset$ (equivalently, $f_0(\CP^3_{\infty}) = \emptyset$). One may, however, obtain $\CP^3_{\infty}$ as a limiting process by letting $|z_1|^2+|z_2|^2+|z_3|^2 \to \infty$ within $\iota(\mathbb{C}^3)$. The relevance of this limiting procedure for the purpose of obtaining a plane-wave analogue for the Laplace-Beltrami equation is explained in Section \ref{s:PlaneWaves}. Using the concept of $\CP^3_{\infty}$, we can now introduce the  ``sphere at infinity'' $\Xi$ as follows: 
    \begin{align}
        \Xi = \hat{\mathcal{S}}^2_{\c} \cap \CP^3_{\infty} = \{\class{0: \eta_1 : \eta_2 : \eta_3} \in \CP^3: ~ \eta_1^2+\eta_2^2+\eta_3^2=0\}. \label{e:revVK2.20}
    \end{align}
\begin{remark}\label{rem:013}
    $\Xi$ can be identified as the quadric surface within $\CP^2$. Thus, $\Xi$ is known to be biholomorphic to the Riemann sphere, see \cite[Chapter 4]{griffiths1984principles}. 
\end{remark}} 
    
 \VK{Let us discuss the properties of $\Xi$ outlined in the previous Remark \ref{rem:013} in more detail, since the identification of $\Xi$ with the Riemann sphere will play an important role throughout the remainder of this article. To this end,} introduce the set of non-zero complex vectors of ``zero length'': 
\begin{equation}
\mathcal{N} = \{ \bdeta = (\eta_1 , \eta_2, \eta_3) \in \mathbb{C}^3 \setminus \{ \boldsymbol{0} \} 
\, \, : \, \, 
\bdeta \cdot \bdeta  = 0
\} .
\label{e:f001a2}
\end{equation}
The dot product of two complex vectors is also defined according to \eqref{e:dot} and does not involve complex conjugation.

\VK{ For brevity, let us introduce the notation
    \begin{align}
         [\bdeta] \equiv \class{0:\eta_1:\eta_2:\eta_3},\quad \text{for} \quad \bdeta=(\eta_1,\eta_2,\eta_3) \in \mathcal{N}.
    \end{align}
We note that for $\bdeta\in\mathcal{N}$, $[\bdeta]$ is automatically in  $\hat{\mathcal{S}}^2_{\c}$ and we can write
    \begin{align}
        \Xi = \{[\bdeta]: \bdeta \in \mathcal{N}\}. 
    \end{align}
}
 In other words, $\Xi$ \VK{can be viewed as} the set of complex directions in the space~$\mathcal{N}$.

The set $\Xi$ can be parametrised as follows. All \VK{points of $\Xi$}, except $[(1 , \pm i , 0)]$,
can be represented as $[\bdeta(\beta)]$ with  
\begin{equation}
\bdeta (\beta)  = (\cos \beta , \sin \beta , i), 
\label{e:bdd001}
\end{equation}
i.e.\ parametrised by a complex angle~$\beta$. The values of $\beta$ belong to the strip 
${\rm Re} [\beta] \in [0, 2 \pi]$, whose sides are attached to each other (the point $\beta$ on the 
left side is attached to $\beta + 2 \pi$ on the right side), see \figurename~\ref{f:shanin_fig_03}, left. 
The classes $[(1 , \pm i , 0)]$ may be understood as being associated with the values $\beta = \pm i \infty$. 
The variable $\beta$ can therefore be seen as a local variable on~$\Xi$. 

To cover the whole of $\Xi$, one can introduce ``patches'' near $[(1, \pm i, 0)]$ with the local coordinates 
\begin{equation}
\tau_{\pm} = e^{\pm i \beta}.
\label{e:xdd001}
\end{equation}
The corresponding parametrisations of points $p\in\Xi$ are then given by 
\begin{equation}
p = \left[ \left( \frac{1 + \tau_{\pm}^2}{2} ,
\pm i \frac{1 - \tau_{\pm}^2}{2} , i \tau_{\pm}
\right) \right].
\label{e:xdd002}
\end{equation}


One can see that $\Xi$ has a complex manifold structure. In fact, topologically, this is a sphere (real dimension 2). For instance, it can be seen as the Riemann sphere by using the variable $\tau_+$ as a global variable.

Throughout the paper, 
we will mainly use the variable $\beta$ as a coordinate on the sphere $\Xi$, allowing the ``values'' 
$\beta = \pm i \infty$ (see \figurename~\ref{f:shanin_fig_03}, right). 
The values $\beta$ and $\beta + 2\pi$ yield the same point of $\Xi$. 
The points $\beta$ and 
$\bar \beta + \pi$ are, in an informal sense, assumed to be opposite on the sphere, as can be seen from Figure \ref{f:shanin_fig_03}, right.\footnote{Another rationale for referring to $\bar \beta + \pi$ as `opposite' to $\beta$ is provided in Section \ref{sec:02.3}.} Here, we used $\bar{ \ }$ to denote complex conjugation. No other geometrical property of the sphere $\Xi$ will be used. 

\begin{figure}[h]
    \centering{\includegraphics[width=0.6\textwidth]{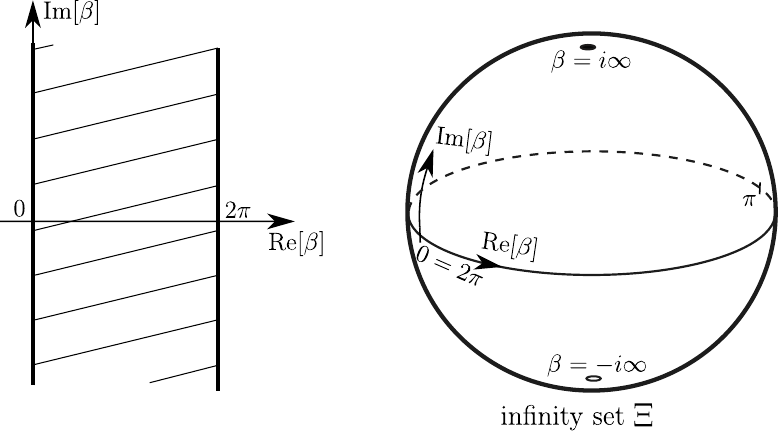}}
	\caption{$\Xi$ as a strip (left) and as a sphere (right)}
	\label{f:shanin_fig_03}
\end{figure}

\sout{We claim that $\Xi$ can be added to $\S^2_{\c}$ as a set of points at infinity, 
forming the compactified complex sphere (or the projective complex sphere) $\hat \S^2_{\c} = \S^2_{\c} \cup \Xi.$}

\renewcommand{\VK}[1]{#1}

Let the points of $\Xi$ be parametrised by (\ref{e:bdd001}) (here we do not consider neighbourhoods of the 
points $\beta = \pm i \infty$).
Let us describe a neighbourhood of $\Xi$ in~$\hat \S^2_{\c}$. To that end, we note that 
\begin{equation}
\tmmathbf{z} = \epsilon^{-1} \left( \cos \beta \, , \, \sin \beta \, , \, i \sqrt{1 - \epsilon^{2}} \right)
\label{e:bdd002}
\end{equation}
belongs to $ \S^2_{\c}$
for any small non-zero complex~$\epsilon$. 
The coordinates $(\epsilon, \beta)$
can be interpreted as local coordinates {of $\iota(\mathcal{S}^2_{\c})$} in the neighbourhood of~$\Xi$. Taking $\epsilon = 0$ in \eqref{e:bdd002} corresponds to the point 
$[(\cos \beta, \sin \beta , i)]$ of~$\Xi$. Indeed, 
    \begin{align*}
        \iota(\z) & = \class{1:\frac{\cos(\beta)}{\epsilon}: \frac{\cos(\beta)}{\epsilon}: \frac{i \sqrt{1-\epsilon^2}}{\epsilon}} \\& =  \class{\epsilon: \cos(\beta): \sin(\beta): i \sqrt{1- \epsilon^2}} \quad \ \ 
        \overset{\epsilon \to 0}{\longrightarrow} ~[(\cos(\beta), \sin(\beta), i)]. \label{e:revVK2.28}\numberthis
    \end{align*}

Near the points $[(1, \pm i, 0)]$, one can use the parametrisation 
(\ref{e:xdd002}) and the proximity variables~$\epsilon_\pm$, such that the neighbourhood of $\Xi$
is parametrised in the coordinates $(\epsilon_{\pm}, \tau_{\pm})$ according to  
\VK{\begin{equation}
\iota(\z) = \class{1:\frac{1 + \tau_{\pm}^2 + \epsilon^2_\pm}{2\epsilon_{\pm}}:
\pm i \frac{1 - \tau_{\pm}^2 -\epsilon_\pm^2}{2\epsilon_{\pm}}: i \frac{\tau_{\pm}}{\epsilon_{\pm}}}.
\label{e:xdd003}
\end{equation}}

\begin{remark} \label{rem:complexvsrealmanifoldSvsXi}
There are two important 2D spheres embedded in $\hat \S^2_{\c}$: 
they are the real sphere $\S^2$ 
and the sphere at infinity~$\Xi$.
They have different properties: $\Xi$ is a complex submanifold of $\hat \S^2_{\c}$, 
while \VK{$\iota(\S^2)$} is not; the latter is a real analytic submanifold. 
These properties can be used as follows. On the one hand, one can continue analytically solutions to the Laplace-Beltrami equation from $\S^2$ to \VK{$\S^2_{\c}$}. On the other hand, due to its complex analytic structure, $\Xi$ will serve as `dispersion surface' for the Laplace-Beltrami equation. Indeed, as we will see in  Sections \ref{s:PlaneWaves} and  \ref{sec:Plane-Wave}, the definition of a plane-wave representation as well as the aforementioned  sliding contours procedure will heavily rely on the complex structure of $\Xi$. This is due to the fact that the contours used within the plane-wave representation are fully restricted onto $\Xi$.
\end{remark}


\subsection{Characteristic lines on the complex sphere}\label{sec:02.3}

\RCA{Here we describe important properties related to the geometry of $\S_{\c}^2$ and $\hat{\S}_{\c}^2$.
A complex line in $\mathbb{C}^3$ passing through a point $\z'$ with direction $\bdeta$, defined as 
\begin{equation}
 \VK{L\left(\z^{\prime}\!; \bdeta\right)=}\left \{\z \in \mathbb{C}^3 : \z = \z' + c \bdeta \text{ for some }
c \in \mathbb{C}\right\}, \label{e:bdd004}
\end{equation}
can lie entirely within $\S_{\c}^2$. In fact, as shown in Appendix \ref{append_A}, the necessary and sufficient condition for this to be the case is to have $\z'\in\S_{\c}^2$, $\bdeta\in\mathcal{N}$, and $\z'\cdot\bdeta=0$. From now on, when writing $L\left(\z^{\prime}\!; \bdeta\right)$ we will always implicitly assume that these conditions are satisfied.
}
The situation resembles a well-known structure of a one-sheeted hyperboloid in 
$\mathbb{R}^3$ formed by two families of real lines (see e.g.\ \cite[p11]{hilbert2021geometry}).

\RCA{Similarly, for any $\class{\uu}\in\CP^3$ and any $\class{\vv}\in\CP^3$ we can define the set
\begin{equation}
    \hat L(\class{\uu};\class{\vv})=\left\{\class{\boldsymbol{w}}\in\CP^3\,:\, \boldsymbol{w}=\uu+c \vv \text{ for some }c\in\mathbb{C}\right\}.
    \label{e:revVK2.31}
\end{equation}
This object lies entirely within $\hat{\S}_{\c}^2$ if and only if $\class{\uu}\in\hat{\S}_{\c}^2$, $\class{\vv}\in\hat{\S}_{\c}^2$, and $u_1v_1+u_2v_2+u_3v_3=u_0v_0$. As before, from now on, when writing $\hat L(\class{\uu};\class{\vv})$ we will always implicitly assume that these conditions are satisfied.

These two sets are linked in the following sense: $L\left(\z^{\prime}\!; \bdeta\right)$ is the finite part of $\hat L(\iota(\z');[\bdeta])$, i.e.\ we can write $L\left(\z^{\prime}\!; \bdeta\right)=f_0(\hat L(\iota(\z');[\bdeta]))$. This can be see directly from the definitions of $f_0$, $\iota$ and $[\bdeta]$.}

For reasons that will become clear in Section \ref{sec:CLBE}, 
a complex line $L\left(\z^{\prime}\!;\bdeta \right)$ fully belonging to $\S^2_{\c}$ will be referred to as a {\em characteristic line} 
(sometimes these lines are also referred to as generators of $\S^2_{\c}$).  
\VK{For brevity, we will also refer to the $\hat{L}(\class{\uu};\class{\vv})$ that are fully contained within $\hat{\S}^2_{\c}$ as characteristic lines.}

Several statements related to characteristic lines are important for this work.
These properties are proven in Appendix~\ref{append_A} and stated below.

\begin{lemma} \label{lem:charac} The following statements are true:
\begin{itemize}
\item[a)]
For each $\z \in \S^2_{\c}$, there are exactly two characteristic lines \VK{$L$ (defined according to \eqref{e:bdd004})} passing through~$\z$. 
\item[b)]
For each \VK{$p \in \Xi$} 
there are exactly two characteristic lines \VK{$\hat{L}$ (defined according to \eqref{e:revVK2.31})} passing through\VK{~$p$}. 

\item[c)] 
The characteristic lines \VK{$\hat L$}
form two connected sets in\VK{~$\hat{\S}_{\c}^2$}. 
A characteristic line passing through a point $q\in\hat \S^2_{\c}$ 
\sout{($P$ can either be some $\tmmathbf{z} \in \S^2_{\c}$ or some $p=[\bdeta] \in \Xi$)}
will be denoted either \VK{$\hat{L}_+(q)$ or $\hat{L}_-(q)$}, depending on the set it belongs to. \VK{If $q\in\hat{\S}_{\c}^2 \setminus \Xi$, the finite part of $\hat{L}_{\pm}(q)$ will be denoted by $L_{\pm}(\z)$ for $\z=\phi_0(q)$.}
The sets \VK{$\hat{L}_{\pm} (q)$} depend continuously on $q$, \VK{and the same statement is true for $L_{\pm}(\z)$.}

\item[d)]
Each characteristic line \VK{$\hat{L}$} 
intersects \VK{$\iota(\S^2)$} at a single point, and it crosses the sphere at infinity $\Xi$ at a single point.

\item[e)] 
Let us consider $\z''\in \S^2_{\c}$. One can show that 
\begin{equation}
L_+ ( \z'' ) \cup L_- ( \z'' ) = 
\{ \z \in \S^2_{\c} 
\, \, : \, \,
\z \cdot \z'' = 1
\} ,
\label{e:f006c}
\end{equation}
i.e.\ the intersection of the sphere $\S^2_{\c}$ and the complex hyperplane defined by $\z \cdot \z'' = 1$
is the union of two characteristic lines.

\item[f)] 
The characteristic lines are rotationally invariant in the following sense. Consider $\z,\z' \in \S^2_{\c}$. If $A \in \mathrm{SO}_{\c}(3) = \{A \in \mathrm{GL}_3(\mathbb{C}): ~ AA^T=1,~\det(A)=1\}$ is a complexified Euler-rotation such that $A\z = \z'$, then $L_{+}(\z') = AL_+(\z)$ and $L_{-}(\z') = AL_-(\z)$.

\end{itemize}

\end{lemma}

A comment should be made about the property~c). Let us specify the families $L_\pm$. 
\VK{This can be done by defining $L_\pm (\z)$ at one particular point~$\z\in \S^2_{\c}$ as follows. }
Let us consider
\[
\z =\xNP=(0 , 0 , 1) \in \S^2,
\]
and define 
\begin{equation}
L_\pm (\z)  = \{  
(0,0,1) + c (\pm 1 , i,  0) 
\, \, : \, \, 
c \in \mathbb{C}
\}.
\label{e:f006c1}
\end{equation}
This defines all characteristics $L_{\pm}(\z)$ by the continuity argument given in property c). \VK{Using the property b) and the fact that $L_{\pm}(\z)$ is the finite part of $\hat{L}_{\pm}(\iota(\z))$, the characteristics $\hat{L}_{\pm}$ are also fully specified by \eqref{e:f006c1}.} An explicit way to obtain the families $L_{\pm}(\z')$ with $\z' \neq \xNP$ is as follows. By property f), we have $L_{\pm}(\z') = A L_{\pm}(\xNP)$ for any complexified Euler-rotation $A$ with $A \xNP = \z'$.\footnote{It can be shown that such rotation always exists, but this does not matter for the purpose of the present article. Hence, we will not discuss it further.}

Take a point $q \in \hat{\S}^2_{\c}$. It follows from the listed properties, that the points of either 
$\Xi$ or $\S^2$ 
can be used to parametrise \VK{the two characteristic lines $\hat L_+(q)$ and $\hat L_-(q)$.}  
These lines cross \VK{$\iota(\S^2)$} at one point each; denote these points \VK{by $\hat{\Psi}_+ (q)$ and $\hat{\Psi}_- (q)$}, 
respectively. 
Similarly, these lines also cross $\Xi$ at one point each; denote these points by
\VK{$\hat{\Phi}_+(q)$ and $\hat{\Phi}_-(q)$}, respectively.  
    \VK{
       In particular, for $p \in \Xi$ and $\x \in \mathcal{S}^2$, we thus define the maps $\Psi_{\pm}(p)$ and $\Phi_{\pm}(\x)$ by 
        \begin{align}
            \Psi_{\pm}(p) = \phi_0(\hat{\Psi}_{\pm}(p))\in \mathcal{S}^2 ~ \text{ and } ~ \Phi_{\pm}(\x) = \hat{\Phi}_{\pm}(\iota(\x)) \in \Xi
        \end{align}
    mapping $\Xi$ onto $\mathcal{S}^2$ and vice-versa.
    }
The maps $\Psi_\pm$, $\Phi_\pm$ are diffeomorphisms. Indeed, 
$\Psi_\pm$ and $\Phi_\pm$ are inverse to each other. 
Explicit formulae for these maps are given by \eqref{e:ap006},
\eqref{e:0229}, \eqref{e:0230} in Appendix~\ref{append_A}.

Next, we introduce the notations
\begin{equation}
\Phi (\x) = \{ \Phi_+ (\x), \Phi_- (\x) \}, 
\label{e:map_Phi}
\end{equation}
\begin{equation}
\Psi(p) = \{ \Psi_+ (p), \Psi_- (p) \}.
\label{e:map_Psi}
\end{equation}
So, $\Psi$ can be interpreted as a function mapping $p=[\bdeta]$ onto the two points of $\S^2$ for which $\x \cdot \bdeta = 0$. Similarly, 
$\Phi$ can be interpreted as function mapping $\x$ to the two points $p = [\bdeta]$ for which  $\x \cdot \bdeta = 0$. For simplicity, we will treat $\Phi$ and $\Psi$ as such maps. They are visualised in Figure \ref{f:02302}.

It is reasonably straightforward to show that there exists a pair-to-pair correspondence, i.e.\ 
if $\x$ and $-\x$ are opposite points of $\S^2$, 
\[
-\x \stackrel{\Phi} \longmapsto \{ \Phi_- (\x) , \Phi_+ (\x) \},
\]
so a pair of points $\{ \x, -\x \}$ is mapped to a pair of points $ \{ \Phi_+ (\x), \Phi_- (\x) \}$ (see \figurename~\ref{f:02302}). 

\begin{figure}[h]
    \centering{\includegraphics[width=0.7\textwidth]{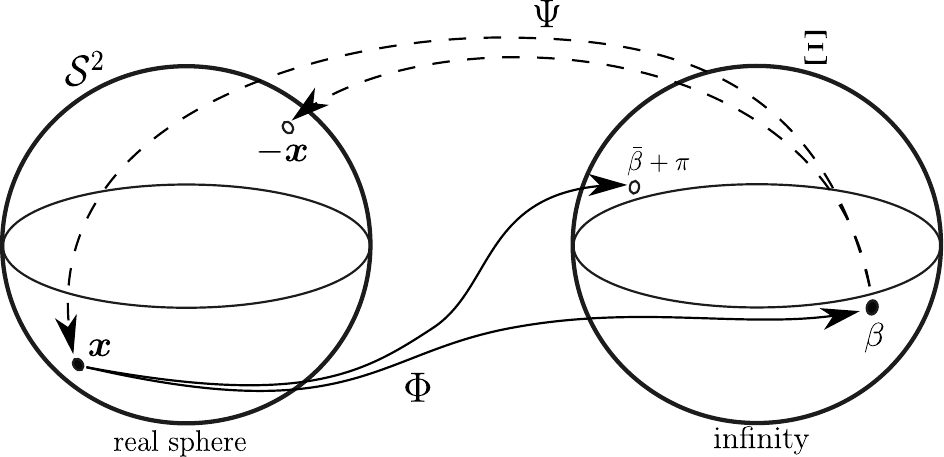}}
	\caption{Illustration of the Maps $\Psi$ and $\Phi$}
	\label{f:02302}
\end{figure}

Let 
\[
\Phi_+ (\x) = [\bdeta (\beta)]  . 
\]
Then, as it follows from \eqref{e:ap006}, or can be checked directly, 
\[
\Phi_- (\x) = [\bdeta (\bar \beta + \pi)]  . 
\]
This means that the map $\Phi$  maps a pair of opposite points on $S^2$ to a pair of opposite points on~$\Xi$.
A similar statement is valid for~$\Psi$.

\begin{remark}
There are two important points $\x\in\S^2$, for which we would like to write down $\Phi(\x)$ explicitly: 
\begin{alignat}{2}
&\xNP \stackrel{\Phi_+}{\longmapsto}  \beta = \i \infty,
\qquad 
&&\xNP \stackrel{\Phi_-}{\longmapsto}  \beta = -\i \infty,
\label{e:ddd002} \\
&\xSP \stackrel{\Phi_+}{\longmapsto}  \beta = -\i \infty,
\qquad 
&&\xSP \stackrel{\Phi_-}{\longmapsto}  \beta = \i \infty.
\label{e:ddd003}
\end{alignat}
This follows from \eqref{e:f006c1} or from \eqref{e:ap006}. 
\end{remark}

    
\subsection{Complexification of the Laplace--Beltrami equation} \label{sec:CLBE}

Let $U$ be an open domain in $\S^2_{\c} \setminus \{ \xNP , \xSP \}$.
Use the complex coordinates \eqref{e:f003} in~$U$. 
A function $u(\z)$, $\z \in U$ is a solution of the complexified Laplace--Beltrami equation
(CLBE) if $u(\theta , \varphi)$ is a holomorphic function of $(\theta, \varphi)$ in $U$, and 
it obeys there the equation having the same form as \eqref{e:02304}:
\begin{equation}
\left( 
\tilde \Delta 
 + \lambda (\lambda + 1)
\right) u (\theta , \varphi) = 0,
\qquad 
\tilde \Delta = 
\frac{\ptl^2 }{\ptl \theta^2} + \frac{\cos \theta}{\sin \theta} \frac{\ptl }{\ptl \theta}+ 
\frac{1}{\sin^2 \theta} \frac{\ptl^2 }{\ptl \varphi^2}.
\label{e:f007}
\end{equation}
Note that the CLBE is homogeneous, and the derivatives are understood in the complex holomorphic sense. In this case, we call $\tilde \Delta$ the complexified Laplace--Beltrami operator (CLBO). As explained in Appendix \ref{app:usefulformulas}, in which we write the CLBO in various coordinate systems, the CLBO cannot be  defined on $\Xi$, so we will only consider it on $\S^2_{\c}$.

The CLBE is a linear second-order PDE, and thus one can define complex characteristic sets in some standard way. One can show that the characteristic lines introduced \VK{in \eqref{e:bdd004}} are exactly the characteristic sets of the CLBE. For example,
using the representation \eqref{e:f008} it is
reasonably straightforward to show, particularly, that the sets 
$L_\pm (\z)$ are characteristics. 
This was of course the reason to call $L_\pm (\z)$ characteristic lines in the first place. \VK{The sets $\hat{L}_{\pm}(q)$ should be viewed as compactified characteristic lines. Indeed, it can be seen that $L_\pm (\z)$ is biholomorphic to $\mathbb{C}$ whereas $\hat{L}_{\pm}(q)$ is biholomorphic to the Riemann sphere $\mathbb{C} \cup \{\infty\}$.}
 
\begin{remark}\label{rem:1.6}
   In analogy with what is generally observed throughout the study of holomorphic PDEs (cf.\! \cite{Shapiro1992TheSF}, \cite{Ebenfelt1994}, or \cite{KhavinsonLundberg2018}), we expect singularities to propagate along the characteristics lines. By this we mean the following. Given a function $f(\x)$ that solves the real Laplace--Beltrami equation and has some singularities on $\S^2$, we expect the singularities of the analytical continuation $f(\z)$ to be located along the characteristic lines emanating from the real singular points of $f(\x)$. 
\end{remark}

An important example of a solution of the CLBE is a complexification of the Green's function (\ref{e:02310}). Namely, take a fixed  
$\z_0 \in \S^2_{\c}$. Then, one can directly check that
\begin{equation}
u(\z) = G(\z; \z_0) = \frac{1}{4 \sin (\pi \lambda)} P_{\lambda} (-\z \cdot \z_0), 
\qquad
\z \in \S^2_{\c}
\label{e:f010}
\end{equation}
is a solution of the CLBE on $\S^2_{\c} \setminus (T(\z_0)\cup T(-\z_0))$, where 
\[
T(\z_0) = \{ \z \in \S^2_{\c} \, \, : \, \, \z \cdot \z_0 =  1\} .
\]
In \eqref{e:f010}, $P_{\lambda}$ is understood as an analytical continuation 
of the Legendre function to $\mathbb{C} \setminus \{ 1,-1 \}$.
According to the property e) of Lemma \ref{lem:charac}, $T(\z_0)\cup T(-\z_0)$ is the union of four characteristic lines: 
\[
T(\z_0) \cup T(-\z_0) = L_+ ( \z_0 ) 
\cup 
L_- (\z_0 ) 
\cup 
L_+ ( -\z_0 ) 
\cup 
L_- ( -\z_0 ).  
\]
The function $u(\z)$ is branching (i.e.\ is multivalued) on~$\S^2_{\c} \setminus (T(\z_0)\cup T(-\z_0))$. 

\begin{remark}
    Let $\z_0 = \x_0 \in \S^2$. Then the function defined by (\ref{e:f010}) is an analytic   continuation of \eqref{e:02310}.
	Generally, some care needs to be taken when analytically continuing in several complex variables since the analytical continuation off a surface is generally not uniquely determined (take $f(z_1,z_2)=z_1z_2$ which is constantly zero on the plane $z_2 \equiv 0$, but not constantly zero within $\mathbb{C}^2$). However, the analytical continuation off a {\em real} analytic surface (like $\S^2$) within its complexification (like $\S^2_{\c}$) is uniquely determined \cite{CieliebakEliashberg2010}. 
    
    In other words, it will be important below that $\S^2$ and $\Xi$ 
    play a different role as submanifolds of $\hat \S^2_{\c}$: $\S^2$ is a real analytic submanifold (i.e.\ not a complex submanifold), 
    and $\Xi$ is a complex submanifold, as already mentioned in Remark \ref{rem:complexvsrealmanifoldSvsXi}. 
    
\end{remark}

\begin{remark}
For more on the importance of complexified PDEs, we refer to \cite{AssierShanin2021AnalyticalCont}.
\end{remark}

\section{A plane wave analogue for the sphere}\label{s:PlaneWaves}

\subsection{Asymptotics of Green's function }\label{sec:GreensAsympt}

To obtain an object that can play the role of a plane wave on $\S^2$, we mimic the procedure known for the planar case: we take a Green's function $G(\z;\z_0)$ and
carry the ``point source'' $\z_0$ to infinity (up to an adequate scaling) \sout{within $\hat \S^2_{\c}$}. The resulting function in the finite part of the space is then declared to be a plane wave coming from the corresponding direction.

As the Green's function, we take the solution $G(\z ; \z_0)$ defined by 
(\ref{e:f010}). \VK{The idea is to ``carry the} source 
\VK{$\z_0$ to infinity'', i.e. to consider a limiting process in such a way that $\iota (\z_0)$ moves to some 
$[\bdeta(\beta)] \in \Xi$} \VK{for}
\begin{equation}
\bdeta(\beta) = (\cos \beta\, , \, \sin \beta \, , \,i) \in \mathcal{N}.
 \label{e:0211a}
\end{equation}
I.e. \VK{$\iota(\z_0)$} approaches the infinity of~$\hat \S^2_{\c}$. The 
point $[\bdeta(\beta)]$ plays the role of the direction of incidence.  
As we noted, $G (\z; \z_0)$ solves the CLBE almost everywhere and, in particular, 
its restriction to $\z  \in \S^2$ solves the real Laplace--Beltrami equation almost everywhere. 
 
Take $\z_0$ parameterised according to \eqref{e:bdd002}:
\begin{equation}
\z_0  = \epsilon^{-1} \left( \cos \beta \, , \, \sin \beta \, , \, i \sqrt{1 -\epsilon^2} \right).
 \label{e:0211}
\end{equation}
\VK{So, informally, we may think of $\z_0$ as being located close to $[\bdeta(\beta)]$.}
The process \VK{$\iota(\z_0)  \to [\bdeta(\beta)]$} is then described by taking the limit $\epsilon \to 0$, \VK{see \eqref{e:revVK2.28}.} 

Take into account the series representation of the Legendre
function $P_\lambda(z)$ near its regular singular point $z = \infty$, valid for $|z|>1$ in the cut plane $\mathbb{C}\setminus(-\infty,0)$, and given by

\begin{align}
	P_{\lambda} (z)  = & A(\lambda) z^{- \lambda - 1} 
	\sum_{n = 0}^{\infty}  a_n(\lambda) z^{- 2n}  
	+ A(-\lambda-1)z^{\lambda} 
	\sum_{n = 0}^{\infty} a_n(-\lambda-1) z^{- 2 n}, 	\label{e:02313}
\end{align} 
where the coefficient functions $A$ and $a_n$ are defined by
\begin{alignat*}{2}
	A(\lambda) &= \frac{2^{- \lambda - 1} \Gamma \left( - \frac{1}{2} -\lambda
		\right)}{\sqrt{\pi} \Gamma (- \lambda)}, \qquad &&a_n(\lambda) =\frac{2^{- 2 n} \Gamma \left(\lambda + \frac{3}{2} \right) \Gamma (2 n + \lambda +
		1)}{\Gamma (\lambda + 1) \Gamma \left( n + \lambda + \frac{3}{2} \right) n!} \cdot 	
\end{alignat*}
This formula can be obtained directly from \cite[Chapter 3]{Bateman1953}.	

Take some $\x \in \S^2$.
Consider the two series in \eqref{e:02313} separately, take the limit $\epsilon \to 0$, and take the leading term in $\epsilon$ in each series. 
This results in
\begin{equation}
	G(\x;\z_0) \sim 
	\frac{1}{4\sin (\pi \lambda)} \left(  
	A_1(\lambda) \epsilon^{1+\lambda} (- \x \cdot \boldsymbol{\eta}(\beta))^{-(1+ \lambda)}
	+ 
	A_1(-\lambda-1) \epsilon^{-\lambda} (-  \x \cdot \boldsymbol{\eta}(\beta) )^{\lambda}
	\right).
	\label{e:02314}
\end{equation}

Omitting the coefficients not depending on $\x$, and bearing in mind the analogy with planar 2D scattering problems, we come to the conclusion that there should be {\em two types\/} of ``plane waves'' obtained from our limiting procedure, namely,
\begin{equation}
	w_1(\x , \beta) =(- \x \cdot \boldsymbol{\eta}(\beta))^{-(1+ \lambda)}
	~ \text{ and } ~
	w_2(\x, \beta) = (- \x \cdot \boldsymbol{\eta}(\beta))^{\lambda}.
	\label{e:02315}
\end{equation}

In the next subsection we will prove that
$w_{1,2}$ (as functions of $\x$) obey the real Laplace--Beltrami equation on $\S^2$ everywhere except for the points 
where 
\begin{equation}
\x \cdot \bdeta (\beta)  = 0.
\label{e:cdd001}
\end{equation}

Note that the functions \eqref{e:02315} depend on our parameterisation of~$\Xi$. 
Namely, \eqref{e:02315} depends on the 
particular vector $\bdeta (\beta) \in \mathcal{N}$, 
while we would like it to depend only on  $[\bdeta]$.
Let us alter the definition \eqref{e:02315} by choosing an appropriate amplitude factor.
    
\begin{definition}[Plane waves]\label{def:planewaves}
	Let $p \in \Xi$ be represented by 
    $p = [\bdeta]$
    for some $\bdeta \in \mathcal{N}$.
    A plane wave for scattering problems to the Laplace--Beltrami equation on $\S^2$ is given by either of the two functions,
	\begin{align}
		w_{1}(\x , p ; \x^{\rm{ref}}) = \left(- i \frac{\x \cdot \bdeta}{\x^{\rm{ref}} \cdot \bdeta }\right)^{-(1+\lambda)}\!\!\! ,
        \qquad  	
        w_{2}(\x , p; \x^{\rm{ref}}) = \left(- i \frac{\x \cdot \bdeta}{\x^{\rm{ref}} \cdot \bdeta }\right)^{\lambda}\!\!\!\!\!, \label{e:0223}
	\end{align}
	where  $\x, \x^{\rm{ref}} \in \S^2$, and $\x^{\rm{ref}}$ is a fixed reference point.

    We choose the branch of these plane waves according to 
	\begin{align}
		w_{1}( \x^{\rm{ref}} , p; \x^{\rm{ref}}) = e^{i \pi (\lambda+1) / 2},
        \qquad 
        w_{2}( \x^{\rm{ref}} , p; \x^{\rm{ref}}) = e^{- i \pi \lambda / 2}. 
        \label{e:0224}
	\end{align}
\end{definition}
    
    The point $p \in \Xi$ is referred to as the wave-vector of the plane wave, and $\Xi$ can be interpreted as the dispersion diagram
    (since it is a surface of all possible wave-vectors). 
    The point $\x \in \S^2$ is referred to as an observation point. Since there are two types of plane waves, the dispersion diagram can also be seen as two copies of $\Xi$.

Note  that the functions $w_j(\x,\beta),~j=1,2,$ given in \eqref{e:02315} satisfy $w_j(\x,\beta)=w_j(\x,[\boldsymbol{\eta}(\beta)]; \xNP)$, since we have
        \begin{align}
            -\x \cdot \boldsymbol{\eta} (\beta)  =  - i \frac{\x \cdot \boldsymbol{\eta} (\beta)}{\x^{\mathrm{ref}} \cdot \boldsymbol{\eta} (\beta)}, ~ \text{ with } ~ \x^{\mathrm{ref}} = \xNP. \label{e:0240}
        \end{align}

	When it is clear which reference point is chosen, or whenever the specific choice does not matter, we shall simply write $w_{1,2}(\x, p)$ instead of $w_{1,2}(\x, p; \x^{\rm{ref}})$. If $p$ is equal to some $[\bdeta(\beta)]\in\Xi$, we might also allow ourselves to write $w_{1,2}(\x, \beta)$ or $w_{1,2}(\x, \beta; \x^{\rm{ref}})$. This statement remains true for any functions of $p$ which can be written as just functions of $\beta$ when appropriate. \VK{Let us also note that the definition of plane-waves is well-defined, i.e., it depends only on the class $p=[\bdeta]$ and not the particular representative $\bdeta$. This can be checked by replacing $\bdeta$ with $c \bdeta$ in \eqref{e:0223} for $c\in \mathbb{C} \setminus \{0\}$.}

One can easily complexify the definition \eqref{e:0223} replacing the arguments ${\x, \x^{\rm ref}}$ by ${\z, \z^{\rm ref}}$, both belonging to~$\S^2_{\c}$: 
\begin{align}
		w_{1}(\z , p ; \z^{\rm{ref}}) = \left(- i \frac{\z \cdot \bdeta}{\z^{\rm{ref}} \cdot \bdeta }\right)^{-(1+\lambda)}\!\!\! ,
        \qquad  	
        w_{2}(\z , p; \z^{\rm{ref}}) = \left(- i \frac{\z \cdot \bdeta}{\z^{\rm{ref}} \cdot \bdeta }\right)^{\lambda}\!\!\!\!\!. \label{e:0223a}
	\end{align}


\subsection{Singularities of plane waves}
\label{sec:03-Sing}

The singularities of plane waves \eqref{e:0223} can be studied in two ways: with respect to $\x$ for a fixed $p = [\bdeta]$, 
and with respect to $[\bdeta]$ for fixed $\x$ and $\x^{\rm ref}$. The singularities with respect to 
$\x$ are given by the following lemma.

\begin{lemma}
\label{le:w_CLB} The following statements are true
\begin{itemize}
    \item[a)] The plane waves \eqref{e:0223} are solutions of the real homogeneous Laplace--Beltrami equation \eqref{e:cdd002} outside the singularity set 
defined by~\eqref{e:cdd001}. This set can be described by~$\Psi(p)$.
\item[b)] The plane waves \eqref{e:0223a} are solutions of the CLBE outside the singularity sets defined by $\z \cdot \bdeta  = 0$, 
i.e.\ outside \VK{$f_0(\hat{L}_+(p) \cup \hat L_- (p)) \subset \mathcal{S}^2_{\c}$}.
\end{itemize}
In both cases, the singularities are of the branching type as explained further in Remark~\ref{rem:024}.
\end{lemma}

The statement of the lemma seems obvious since the plane waves are obtained in the course of a limiting procedure 
in which the Green's function $G(\z ; \z_0)$ obeys the CLBE or the real Laplace--Beltrami equation. 

One can also prove the lemma explicitly by considering  $w_j$ in the form
\eqref{e:02315}, and taking $\bdeta = (1 , \pm i ,0)$.
In the complex coordinates $(\theta, \vph)$ the plane waves can be written as 
\begin{align}
    w_1 = (- \sin \theta \, e^{\pm i \vph} )^{-(1 + \lambda)}, 
\qquad 
w_2 = (- \sin \theta \, e^{\pm i \vph} )^{\lambda}. \label{e:043}
\end{align}

One can check that these expressions obey  \eqref{e:f007} if $\sin \theta \ne 0$. 
Any other choice of $\bdeta$ can be obtained from  $\bdeta = (1 , \pm i ,0)$ by 
a rotation of the coordinates.    

\begin{remark}\label{rem:024}
        The branching of $w_1(\z,p;\z^{\mathrm{ref}})$ and $w_2(\z,p;\z^{\mathrm{ref}})$ in the $\z$-variable is as follows. Let $\sigma_{+}\subset \S^2_{\c}$ denote a loop \VK{such that $\iota(\sigma_+)$} encircles \VK{$\hat{L}_{+}(p)$} exactly once, and let $\sigma_{-}\subset \S^2_{\c}$ denote a loop \VK{such that $\iota(\sigma_-)$} encircles \VK{$\hat{L}_{-}(p)$} exactly once. \VK{These curves are oriented} such that the curves described by \VK{$\hat{\Phi}_+(\iota(\sigma_+))$} and \VK{$\hat{\Phi}_-(\iota(\sigma_-))$} encircle $p\in \Xi$ in the positive direction (on $\Xi$). Then, the branching of $w_1(\z,p;\z^{\mathrm{ref}})$ and $w_2(\z,p;\z^{\mathrm{ref}})$ is according to
            \begin{align}
                w_1(\z,p;\z^{\mathrm{ref}}) \overset{\sigma_{\pm}}{\longrightarrow} e^{-2\pi i (1+\lambda)} w_1(\z,p;\z^{\mathrm{ref}}) ~ \text{ and } ~  w_2(\z,p;\z^{\mathrm{ref}}) \overset{\sigma_{\pm}}{\longrightarrow} e^{2\pi i \lambda} w_2(\z,p;\z^{\mathrm{ref}}). \label{e:045}
            \end{align}
         The correctness of \eqref{e:045} can be verified directly from the explicit formulae given in \eqref{e:043}, by taking $\sigma_{\pm}$ to be curves in the complex $\theta$-plane encircling $\theta=0$ exactly once, in the positive direction. 
    \end{remark}

The singularities with respect to $[\bdeta]$ are given by 
another lemma:

\begin{lemma}
\label{le:sing_w}
For fixed $\x, \x^{\rm ref} \in \S^2$, the functions $w_{1,2} (\x , p ; \x^{\rm ref})$ (considered as functions of $p=[\bdeta] \in \Xi$) 
are holomorphic on $\Xi$ outside the singularity  set defined by 
\begin{equation}
	\x \cdot \bdeta =0 ~ \text{ or } ~ \x^{\rm ref} \cdot \bdeta =0,
    \label{e:0217}
\end{equation}
i.e.\ outside the set $\Phi (\x) \cup \Phi(\x^{\rm ref})$.

The singularity is of the branching type: for a local complex  variable $\tau$ centered at one of the points  $\Phi_{\pm} (\x)$
the function $w_1$ behaves as $\tau^{-\lambda-1}$, while $w_2$ behaves as $ \tau^{\lambda}$ as $\tau\to0$.
Similarly, for a local complex variable $\tau$ centered at one of the points  $\Phi_{\pm} (\x^{\rm ref})$
the function $w_1$ behaves as $\tau^{\lambda+1}$, while $w_1$ behaves as $\tau^{-\lambda}$ as $\tau\to0$.
\end{lemma}

The statement of the lemma can be checked, say, by using the parametrisation ${p=[\bdeta(\beta)]}$.

    \begin{remark}
            If the $\beta$-parametrisation of $\Xi$ is chosen such that $\Psi(\x^{\mathrm{ref}})=\{i \infty, -i \infty\}$, then the condition $\x^{\mathrm{ref}} \cdot \boldsymbol{\eta} =0$ gives exactly the singular points $\beta= \pm i \infty$ on $\Xi$. This can be seen from \eqref{e:ddd002} upon choosing $\x^{\mathrm{ref}} = \xNP$, and it is confirmatory of the $\beta$ singularities appearing in the representation \eqref{e:02315}. For general $\x^{\mathrm{ref}}$, the statement of this remark can be verified by using the rotational invariance of characteristics.
         \end{remark}



\section{Plane-wave representation of a field}
\label{sec:Plane-Wave}

\subsection{Structure of a plane-wave representation}


The construction below heavily exploits the fact that the infinity set $\Xi$ is a complex manifold having the structure 
of a Riemann sphere. A choice of the local coordinate (it may be, say, \eqref{e:bdd001} or \eqref{e:xdd002}) 
can be made for each particular problem. 

Let $u(\x)$ be some function of $\x \in \S^2$, and $\x^{\rm ref}$ be some reference point on~$\S^2$;
By a ``plane-wave representation'' of $u$, we refer to any integral representation of the form 
\begin{align}
	u(\x) = \int_{\gamma_1} A_1(p) \, w_1(\x,p; \x^{\rm ref}) \, \Omega + \int_{\gamma_2} A_2(p) \, w_2(\x,p; \x^{\rm ref})\,\Omega, 
    \label{e:0436}
\end{align}
where 
\begin{itemize}
	
    \item $A_{1,2} (p)$ are {\em spectral functions\/} defined on $\Xi$ and holomorphic in a certain domain there.
	
    \item $\gamma_{1,2} \subset \Xi$ are smooth oriented contours of integration on $\Xi$. We use closed contours (loops).
	
    \item $\Omega = \Omega(p)$ is some meromorphic 1-form on~$\Xi$. 
    
\end{itemize}

The representation \eqref{e:0436} is very general. To specify it, 
we can choose the coordinate $\beta$ on $\Xi$ according to \eqref{e:bdd001}. 
A reasonable choice for $\Omega$ is then 
$\Omega = d \beta$, and \eqref{e:0436} reads as 
\begin{equation}
u(\x) = \sum_{j = 1}^2 \int_{\gamma_j} A_j(\beta) \, w_j (\x , [\bdeta (\beta)] ; \x^{\rm ref}) \, d\beta. 
 \label{e:0436a}
\end{equation}

Let us study the form $\Omega = d\beta$ on~$\Xi$. The form is holomorphic everywhere except at the points 
$\beta = \pm i \infty$. Consider the point $\beta = -i \infty$. Introduce the variable $\tau_- = e^{- i \beta}$ as in \eqref{e:xdd001}. 
One can see that $\tau_- = 0$ for $\beta = - i \infty$, and 
\[
d\beta = i \frac{d\tau_-}{\tau_-}.
\]
Thus, $\Omega$ has a simple pole at $\beta = -i \infty$ with residue equal to~$i$.
Similarly, one can study $\Omega$ at $\beta = i \infty$ by introducing the local variable 
$\tau_+ = e^{i \beta}$. The form $\Omega$ has a simple pole there with residue equal to~$-i$.

Let us discuss some properties of the decomposition \eqref{e:0436}. 
By applying Lemma~\ref{le:w_CLB}, we can directly obtain the following statement: 

\begin{lemma}
\label{le:LB_u}
The field $u(\x)$ defined by \eqref{e:0436} obeys the real Laplace--Beltrami equation
\eqref{e:cdd002} for $\x \in \S^2 \setminus \Psi (\gamma_1 \cup \gamma_2)$. 
\end{lemma}

We are going to use Cauchy's theorem for the representation \eqref{e:0436}. The following lemma will be used for this:
\begin{lemma}
\label{le:hol_u}
Let $U_j \subset \Xi$, $j = 1,2$, be open domains in $\Xi$, such that the forms $A_j (p) \Omega(p)$ are holomorphic for $p \in U_j$.
Then the integrand of the $j$th term of \eqref{e:0436}, i.e.\ the form 
\[
\omega_j(p) =  A_j (p) \, w_j (\x , p ; \x^{\rm ref}) \, \Omega(p),
\]
is holomorphic in the domain $U_j \setminus (\Phi(\x) \cup \Phi(\x^{\rm ref}))$.
\end{lemma}
This statement follows from Lemma~\ref{le:sing_w}.

Let us explain how Lemma~\ref{le:hol_u} will be used. Let $u(\x)$ be described by \eqref{e:0436} in some domain $X \subset \S^2$. 
Then for all such $\x$ the form $\omega_j(p)$ is holomorphic in 
\[
U_j' (X) =  U_j \setminus (\Phi(X) \cup \Phi(\x^{\rm ref})) \subset \Xi.
\]
Thus, by Cauchy's theorem, one can freely deform the contour of integration $\gamma_j$ within $U_j' (X)$ such that the values $u(\x)$ do not change. For a deformed contour, the region of validity of the representation  \eqref{e:0436a} generally changes from $X$, to another domain $X' \subset \S^2$. Thus, we introduce the concept of a \textit{sliding contour} that helps us 
to extend the validity of \eqref{e:0436a} to a wider region on~$\S^2$. We will now exhibit this process by considering the specific case of the Green's function $G(\x;\x_0)$.


\subsection{A first plane-wave representation of the Green's function}
\label{sec:PWGreen}

Consider the well-known formula for the Legendre function $P_{\lambda}$ \cite[Chapter~3]{Bateman1953}: 
\begin{equation}
	P_\lambda (q)
	=
	\frac{1}{2\pi}
	\int_{-\pi}^{\pi}
	\left( 
	q + i \sqrt{1- q^2} \cos \alpha
	\right)^\lambda
	d\alpha, 
	\qquad
	0<q\leq1.
	\label{e:02323b}
\end{equation}
This formula can be converted into a plane-wave representation of the Green's function $G(\x ; \x_0)$ as is formalised in the following lemma.

\begin{lemma}\label{Lem:0441}
For $\arccos(\x \cdot \xNP) > \pi/2$ and $\bdeta(\beta)$ given by \eqref{e:bdd001}, we have 
\begin{equation}
		G(\x;\xNP) = \frac{(-i)^{-\lambda}}{8 \pi \sin (\pi \lambda)} \int_{\gamma}
			w_2(\x,[\bdeta(\beta)]; \xSP) \, d \beta  ,
\label{e:02324}
\end{equation}
	where the contour $\gamma \subset \Xi$ is the oriented segment $[-\pi , \pi]$ in the variable~$\beta$ (a loop on $\Xi$ given the $2\pi$-periodicity of the $\beta$-plane).   
\end{lemma}

\begin{proof}

Parametrise $\x$ by \eqref{Ch02.eq.01} and note that the condition $\arccos(\x \cdot \xNP) > \pi/2$ implies that $\pi/2<\theta\leq\pi$, so that $\cos(\theta)<0$ and $\sin(\theta)\geq0$.
Use \eqref{e:02309} and \eqref{e:02323b} to rewrite $G(\x ; \xNP)$ in the form 
\[
G(\x ; \xNP)  = \frac{1}{8 \pi \sin (\pi \lambda)} 
\int_{-\pi}^\pi \left(   
-\cos \theta + i \sin \theta \, \cos \alpha 
\right)^\lambda \, d \alpha, \qquad \pi/2 <\theta \le \pi. 
\]
Using the definition \ref{def:planewaves} of plane waves and the definition \eqref{e:bdd001} of $\bdeta(\beta)$, one can directly check that 
\[
w_2(\x, [\bdeta(\beta)]; \xSP) = 
( \boldsymbol{\eta}(\beta) \cdot \x)^{\lambda} =(-i)^{\lambda}\left(- \cos \theta +  i\sin \theta \, \cos (\beta - \vph)\right)^\lambda.
\]
Setting $\alpha = \beta - \vph$ and using the periodicity of the integrand with respect to $\alpha$, obtain~\eqref{e:02324}.

\end{proof}

One can see that \eqref{e:02324} is a plane-wave representation of $G(\x ; \xNP)$ with
 
\[
A_1 \Omega = 0, \qquad A_2 \Omega = \frac{(-i)^{-\lambda} d\beta}{8 \pi \sin (\pi \lambda)}. 
\]

The expression describes $G(\x ; \xNP)$ in the lower hemisphere $\theta > \pi/2$. Note that we have placed the reference point $\x^{\rm ref} = \xSP$
in this hemisphere as well. This makes the  definition of the branch of the plane waves straightforward.  

According to the explicit formulae \eqref{e:0229}, \eqref{e:0230}, the image $\Phi(\gamma)$ is the equator $\theta = \pi/2$
on~$\S^2$. Thus, according to Lemma~\ref{le:LB_u}, the expression \eqref{e:02324} provides a field obeying the real Laplace--Beltrami 
equation for $\theta > \pi/2$. 

\begin{remark}
It may seem that we haven't introduced anything new so far, just represented the known formula \eqref{e:02323b} in the new notations. 
However, there is a bit more to it. Indeed, if we look for the solution in the form of the plane-wave representation 
\[
G(\x ; \xNP) = \int_{\gamma} A_2(\beta) \, w_2 (\x , [\bdeta(\beta)] ; \xSP)
 \, d\beta,
 \]
we can guess from a symmetry argument that $A_2 (\beta)$ is a constant, so the representation \eqref{e:02324} 
becomes known, at least up to a constant coefficient. 
\end{remark}


\subsection{Extension of the plane-wave representation by sliding contours} \label{sec:slidingcontours}

Let us use the benefits of the complex structure on $\Xi$ and the possibility to deform the integration contour. 
Our aim is to extend the area of validity of the representation \eqref{e:02324} using the concept of {\em sliding contours}. 
We take our inspiration from the planar 2D case (see Appendix~\ref{app:planar}).
We are going to build a representation of the Green's function $G(\x ; \xNP)$ for $\theta > \mu$, where
$\mu$ is an arbitrary small positive constant. 

The idea of the sliding contours is as follows. The domain 
\[
X = \{ (\theta, \vph) \in \S^2 \, \, : \, \, \mu <\theta \le \pi \}  
\]
is so ``large'' that
a single plane-wave representation cannot be built for it. Thus, we are looking for several plane-wave representations 
whose areas of validity cover this domain, and which are continuations of each other. 
Namely, we look for several  open domains $X^{(m)} \subset \S^2$, contours of integration $\gamma^{(m)}$, and spectral 
functions $A_2^{(m)} (\beta)$, such that 

\begin{itemize}

\item[1)] 
The union of all $X^{(m)}$ covers $X$.

\item[2)]
The following plane-wave representations are valid
\begin{equation}
 G(\x ; \xNP) = u^{(m)}(\x) = \int_{\gamma^{(m)}} A_2^{(m)}(\beta) \, w_2 (\x , [\bdeta(\beta)]) ; \xSP ) \, d\beta, 
\qquad 
\x \in X^{(m)}.
\label{e:cdd005}
\end{equation}
This implies that 
\begin{equation}
X^{(m)} \cap \Psi ( \gamma^{(m)} )  = \emptyset.
\label{e:cdd006}
\end{equation}

\item[3)]
For non-empty intersections $X^{(m)} \cap X^{(n)}$ one can convert the plane-wave representation of $u^{(m)}(\x)$
into that of $u^{(n)} (\x)$ by using Cauchy's theorem. Namely, for this set it is possible to deform 
$\gamma^{(m)}$ into $\gamma^{(n)}$ in such a way that the representation remains valid during the deformation.
This ensures that the representations $u^{(m)} (\x)$ and $u^{(n)} (\x)$ match in the intersection domain, and
each of them is a continuation of the other in the sense of solutions of the Laplace--Beltrami equation. 

\end{itemize}

Rigorously, the property 3)  can be formulated as follows: there exists a domain $\Lambda^{(m,n)} \subset \Xi$ such that 
\[
\left\{
\begin{aligned}
  &\ptl \Lambda^{(m,n)}  = \gamma^{(m)} - \gamma^{(n)}, \\
  &A_2^{(m)} (\beta) = A_2^{(n)} (\beta) \text{ for } \beta \in \Lambda^{(m,n)},\\
  &\Lambda^{(m,n)} \cap \Phi(X^{(m)} \cap X^{(n)})  = \emptyset ,\\
  &A^{(m)}_2 d\beta \text{ is holomorphic in }\Lambda^{(m,n)},
\end{aligned}
\right.
\]
where, as is commonly used, by $-\gamma$ we mean the contour $\gamma$ with the opposite direction, and by the sum of two contours we mean that the integral over the sum is the sum of the integrals.

Let us apply the idea of sliding contours to extend the representation \eqref{e:0436a}.
Take 
\[
\gamma^{(1)} = \gamma, 
\qquad
X^{(1)} =  \{ (\theta, \vph) \in \S^2 \, \, : \, \, \theta > \pi/2 \},   
\qquad
A_2^{(1)} (\beta) = \frac{(-i)^{-\lambda} }{8 \pi \sin (\pi \lambda)},
\]
and try to find several other domains and contours obeying the properties 1) -- 3). 

Since any $A^{(m)}_2$ is a continuation of $A^{(1)}_2$, and $A^{(1)}_2$ is a constant
with respect to $\beta$, we can conclude that 
\[
A_2^{(m)} = \frac{(-i)^{-\lambda} }{8 \pi \sin (\pi \lambda)}
\]
for all $m$, so it is only necessary to find the domains $X^{(m)}$ and the contours~$\gamma^{(m)}$.

Let us take the following sets $X^{(m)}$, $m = 2,\dots , 5$:
\begin{align*}
    X^{(2)}  &= \{(x_1, x_2 , x_3) \in \S^2 \, \, : \, \, x_3 <  x_1 / \delta\}, \\
    X^{(3)}  &= \{(x_1, x_2 , x_3) \in \S^2 \, \, : \, \, x_3 < x_2 / \delta\}, \\
    X^{(4)}  &= \{(x_1, x_2 , x_3) \in \S^2 \, \, : \, \, x_3 < -x_1 / \delta\}, \\
    X^{(5)}  &= \{(x_1, x_2 , x_3) \in \S^2 \, \, : \, \, x_3 < - x_2 / \delta\}.
\end{align*}
 The small positive parameter $\delta$ is chosen in such a way that 
the union of all $X^{(m)}$ covers~$X$. The choice $\delta=\mu/2$ would work, for example.
The sets $X^{(m)}$, $m = 2, \dots, 5$ lie below 
the big circles $C^{(m)} = \ptl X^{(m)}$(with respect to the ``altitude'' $x_3$). For instance, we have
\[
C^{(2)}  =  \{(x_1, x_2 , x_3) \in \S^2 \, \, : \, \, x_3 =  x_1 / \delta\}.
\]
We also define the set $C^{(1)}$ to be the equator of~$\S^2$. The big circles $C^{(m)}$ are illustrated in \figurename~\ref{f:shanin_02}.

\begin{figure}[h]
    \centering{\includegraphics[width=\textwidth]{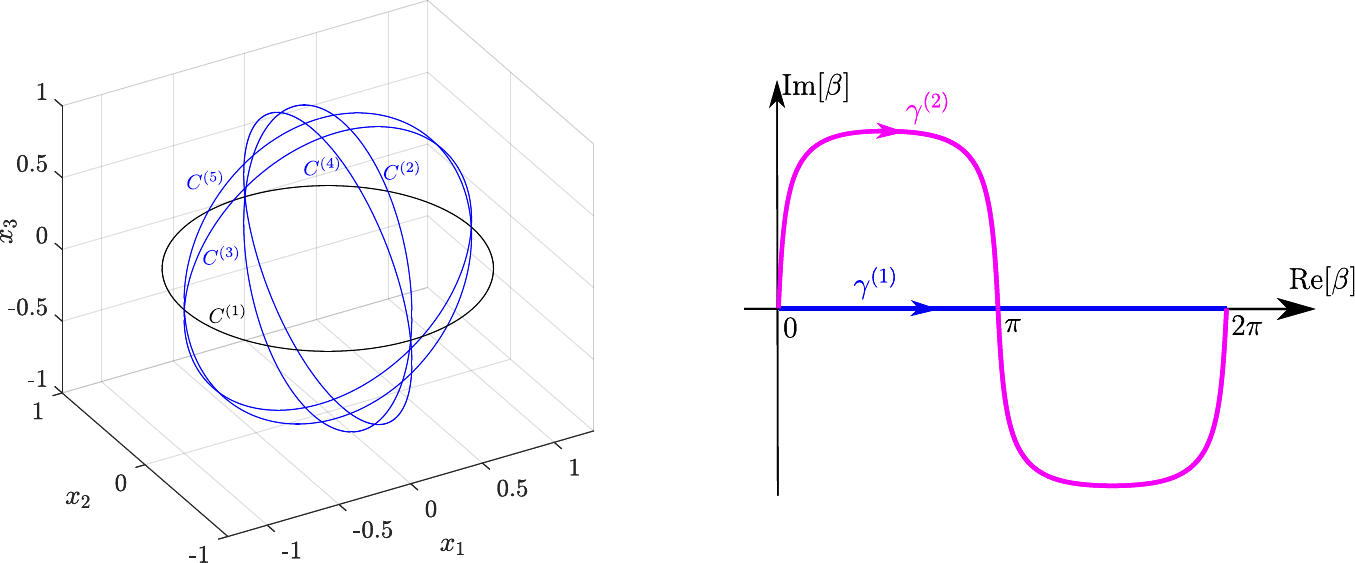}}
	\caption{Contours $C^{(m)}$ on $\S^2$ (left); Contours $\gamma^{(1)}$ and $\gamma^{(2)}$ on $\Xi$ (right)}
	\label{f:shanin_02}
\end{figure}

Define the contours $\gamma^{(m)}$ as follows: they are formed by the points belonging 
to $\Phi (C^{(m)})$ and oriented as shown in \figurename~\ref{f:shanin_02}, right.
We only display the contours $\gamma^{(1)}$ and $\gamma^{(2)}$; 
the contours $\gamma^{(3)}, \gamma^{(4)}, \gamma^{(5)}$ are obtained from 
$\gamma^{(2)}$ by shifting it by $\pi / 2, \pi , 3\pi/2$ to the right (and use periodicity to display them on the $(0,2\pi)$ strip in the $\beta$-plane).

Let us demonstrate that the property 3) is valid. As an example, consider the intersection 
$X^{(1)} \cap X^{(2)}$ (see \figurename~\ref{f:shanin_05}, left). The field in this 
domain should be described by \eqref{e:cdd005} for $m = 1$ and $m = 2$. 
To prove that these representations are identical, we should show that one can 
deform $\gamma^{(1)}$ into $\gamma^{(2)}$ for $\x \in X^{(1)} \cap X^{(2)}$. 

Note that $\x^{\rm ref}  = \xSP \in X^{(1)} \cap X^{(2)}$. Besides, the points 
$\beta = \pm i \infty$ (the poles of $\Omega$) belong to $\Phi(X^{(1)} \cap X^{(2)})$.

Let us display $\Xi$ as the sphere parametrised by $\beta$, as we did in \figurename~\ref{f:shanin_fig_03}. 
On \figurename~\ref{f:shanin_05}, the contours $\gamma^{(1)}$ and $\gamma^{(2)}$ are shown by blue and magenta lines respectively, and the 
domain ${\Phi (X^{(1)} \cap X^{(2)})}$ is shaded. We take the unshaded domain as~$\Lambda^{(1,2)}$.  
According to Lemma~\ref{le:hol_u}, the integrand of \eqref{e:cdd005} is holomorphic in 
$\Xi \setminus \Phi (X^{(1)} \cap X^{(2)})$. Therefore, the contour $\gamma^{(1)}$ can be deformed into 
$\gamma^{(2)}$ without hitting the singularities. So, according to Cauchy's theorem, 
the fields $u^{(1)} (\x)$ and~$u^{(2)} (\x)$ are identical in $X^{(1)} \cap X^{(2)}$.
The other intersections of the domains $X^{(m)}$ are studied in the same way. 

Therefore, we established that $G(\x; \xNP)$, $\theta > \mu$, is given by a plane-wave representation of the form
(\ref{e:cdd005}) with a system of sliding contours $\gamma^{(1)}, \dots , \gamma^{(5)}$. 

A generalisation of these sliding contours plane-wave representation formulae giving $G(\x;\x_0)$ for some arbitrary point source location $\x_0$ is provided in Appendix~\ref{app:sym}. In this same appendix, we also explain why this representation entails the expected reciprocity property of the Green's function, as well as the $\lambda$-symmetry property of the Legendre function. 

\begin{figure}[h]
    \centering{\includegraphics[width=0.8\textwidth]{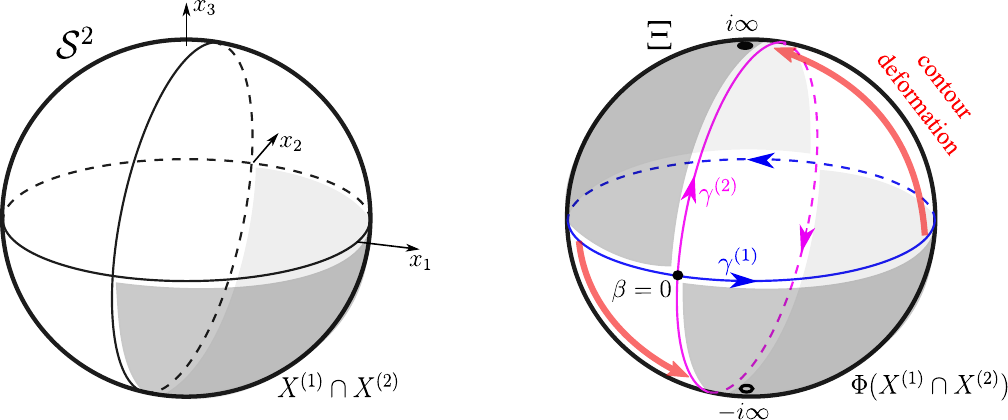}}
	\caption{Domain $X^{(1)}\cap X^{(2)}$ (left); 
    deformation of $\gamma^{(1)}$ into $\gamma^{(2)}$ on $\Xi$ (right)}
	\label{f:shanin_05}
\end{figure}


\begin{remark}
Let  $\x \in X^{(1)}$. The points $\Phi_+ (\x)$ and $\Phi_-(\x)$ are separated by $\gamma^{(1)}$. 
Moreover, the points $\Phi_+ (\x^{\rm ref})$ and $\Phi_-(\x^{\rm ref})$ are also separated 
by~$\gamma^{(1)}$.
Consider $w_2 (\x , [\bdeta(\beta)] ; \xSP)$ as a function of $\beta$ on $\Xi$.
One can see that this function has branching of order $\lambda$ at $\Phi_{\pm} (\x)$, 
and branching of order $-\lambda$ at $\Phi_{\pm} (\x^{\rm ref})$ (see also Lemma \ref{le:sing_w}, where this is discussed).
Thus, one can connect the corresponding points by cuts (see \figurename~\ref{f:shanin_06}),
making the function single-valued everywhere away from the cuts. This reasoning ensures that the 
integral \eqref{e:cdd005} is defined correctly.

\begin{remark}
    The continuation of the plane-wave representation obtained from the sliding contour process is unique. This is due to the fact that $G(\x;\xNP)$ is real analytic in $\x \in \S^2 \setminus \{\xNP\}$, so every plane-wave representation given in \eqref{e:cdd005} is real-analytic in its domain, and they are mutual real-analytic continuations of each other. Just as in the complex-analytic setting, real analytic continuations are uniquely determined by their values on an open domain \cite{KrantzParks2002Primer}.
\end{remark}

\begin{figure}[h]
    \centering{\includegraphics[width=0.3\textwidth]{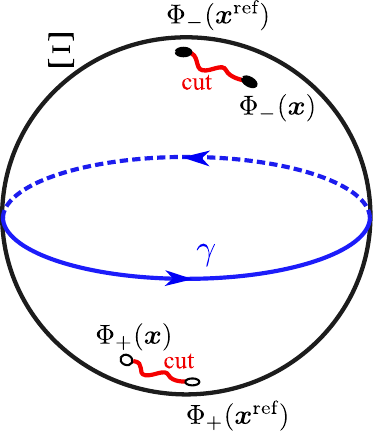}}
	\caption{Contour and singularities associated to \eqref{e:02324} for a given $\x$}
	\label{f:shanin_06}
\end{figure}
\end{remark}

\begin{remark}
In \eqref{e:02324}, we give a plane-wave representation with reference point $\x^{\text{ref}}=\xSP$ for the Green's function with point source $\x_0=\xNP$. Since $\xNP$ and $\xSP$ are antipodal, we have $\Phi(\x_0)=\Phi(\xNP)=\Phi(\xSP)=\Phi(\x^{\text{ref}})$. Therefore, the singularities of the plane-wave representation \eqref{e:02324} are fully described by the set $s(\x) = \Phi(\x) \cup \Phi(\x^{\text{ref}})$, as displayed in Figure \ref{f:shanin_06}. According to Cauchy's theorem, we may continuously deform the contour $\gamma$ of integration on $\Xi$ so long as none of these singularities are hit. Conversely, as $\x$ varies on $\S^2$, we should continuously change $\gamma$ to prevent the singular set $s(\x)$ from hitting the contour of integration. This is exactly what is done in the sliding contours procedure. 
We can extend this idea to complex $\z$, i.e.\ introduce $s(\z)$ in the same way
and study the evolution of the contour for the analytic continuation of~$G(\z ; \xNP)$. 
This will result in the integral representation for an analytical continuation of the Legendre function. If one wants to consider the re-written representation given by \eqref{e:ddd001}, for which we consider an arbitrary source $\x_0\in\S^2$ with an antipodal reference point $\x^{\text{ref}}=-\x_0$, we instead consider $s(\x) = \Phi (\x) \cup \Phi(\x_{0})$. For general plane-wave representations of the type \eqref{e:0436}, we consider $\tilde{s}(\x) = \Phi(\x) \cup \Phi(\x_0) \cup S_1 \cup S_2$ where $S_1$ and $S_2$ constitute the singularities of the forms $A_1(p)\Omega$ and $A_2(p)\Omega$ within \eqref{e:0436}, respectively. 
\end{remark}

\FloatBarrier


\section*{Conclusions}
 We considered wave fields governed by the Laplace--Beltrami equation on the real sphere $\S^2$. We discussed the complexification of this equation onto the complexified sphere $\S^2_{\c}$, and we discussed how $\S^2_{\c}$  is compactified by adding a ``sphere at infinity'' $\Xi$ onto it.  $\Xi$ has the complex manifold structure of a Riemann sphere. 
The main results of this article can be summarised as follows: 

\begin{itemize}

\item We introduced an analogue of plane waves for the Laplace--Beltrami equation considered. These plane waves, given in Definition \ref{def:planewaves}, are relatively simple (elementary) functions obeying the Laplace--Beltrami equation almost everywhere. There are two types of plane waves, and the set of all plane waves (the dispersion diagram) can be seen as the union of two copies of the sphere at infinity $\Xi$.


\item Based on this, we proposed the general formula \eqref{e:0436} for plane-wave representation of wave fields on $\S^2$. This representation involves complex contour integrals on  $\Xi$.  

\item
Furthermore, the concept of sliding contours was introduced for these plane-wave representations. This 
concept enables one to extend the domain of validity of a single plane-wave integral representation. The key feature of  the sliding contours representation is the possibility to use Cauchy's theorem on  $\Xi$ to continuously deform the integration contours.

\item 
Throughout the article, the proposed technique was illustrated by the simplest example: the Green's function of the entire sphere. This resulted in a formal way to extend the validity of the well-known integral representation given by \eqref{e:02309} and \eqref{e:02323b} from the lower hemisphere to the whole sphere save for the Green's function's logarithmic singularity. 

\end{itemize}

The concepts and results of this article pave the way for several other investigations. In particular, we plan to use them to study wave scattering problems on the real sphere (when parts of the sphere are subject to specific boundary conditions).


\vskip2pc

\bibliographystyle{apalike}
\bibliography{bibliography}


\appendix
\numberwithin{equation}{section}


\section{Some properties of the characteristic lines}
\label{append_A}

The aim of this appendix is to prove the properties a)--f) of Lemma \ref{lem:charac} summarising important points regarding characteristic lines.
Consider a complex line \VK{$L(\z';\bdeta)$} as defined in \eqref{e:bdd004}.  Since 
\[
(\z' + c \tmmathbf \eta)\cdot(\z' + c \tmmathbf \eta)
= \z' \cdot \z' + c^2 \bdeta \cdot \bdeta
+ 2c \,\bdeta \cdot \z' ,
\]
we conclude that \RCA{this line lies entirely within $\S_{\c}^2$ if and only if this quantity is equal to one for all $c$, i.e. if and only if we have} 
\begin{align}
    \tmmathbf{z}' \cdot \tmmathbf{z}' &= 1, \label{e:ap001} \\
\bdeta \cdot \bdeta &= 0, \label{e:ap002} \\
\bdeta \cdot \z' &= 0. \label{e:ap003}
\end{align}

The first equation means that $\z' \in \S^2_{\c}$, and the second equation means that 
$\bdeta \in \mathcal{N}$. 
The system (\ref{e:ap001})--(\ref{e:ap003}) is all that is needed to investigate the characteristic lines.
To find the characteristic lines passing through a finite point $\z'\in \S^2_{\c}$, 
one should solve (\ref{e:ap002}), (\ref{e:ap003}) for~$\bdeta$; 
to find the characteristic lines passing through an infinite point $[\bdeta]$, 
one should solve (\ref{e:ap001}), (\ref{e:ap003}) for~$\z'$.
In both cases, the system is underdetermined; therefore, in the first case, the solution is defined up to 
a  non-zero constant factor $c\in\mathbb{C}$, and in the second case, the solution is defined up to an additive term $\alpha \bdeta$ for some $\alpha\in\mathbb{C}$. \VK{Note that any pair $(\z';\bdeta)$ obtained from solving the system (\ref{e:ap001})--(\ref{e:ap003}) determines exactly one complex line in $\hat{\mathcal{S}}^2_{\c}$, and this line is given by $\hat{L}\left(\iota(\z'); [\bdeta]\right)$. The corresponding line $L\left(\z^{\prime};\bdeta \right)$ within $\mathcal{S}^2_{\c}$ is determined according to \eqref{e:revVK2.31}.}

Let us find $\bdeta$ for a given $\z'\in \S^2_{\c}$ parametrised by the
complex variables $(\theta, \vph)$: 
\begin{equation}
\z' (\theta, \vph) = ( \sin \theta \, \cos \vph  \, , \,
\sin \theta \, \sin \vph \, , \,
\cos \theta ) .
\label{e:ap003a}
\end{equation}
Solving (\ref{e:ap002}) and (\ref{e:ap003}), and using the parametrisation \eqref{e:bdd001}, we obtain 
\begin{equation}
	\boldsymbol{\eta}_{\pm}(\z') =  c \left(-i \cos \theta \, \cos \varphi \mp \sin \varphi
    \, , \,
    - i \cos \theta \, \sin \varphi \pm \cos \varphi
    \, , \, 
    i\sin \theta \right), 
    \qquad 
    c \in \mathbb{C} \setminus \{ 0 \}.
\label{e:ap004}
\end{equation}	
The properties a) and f) can be checked explicitly from \eqref{e:ap004}.
One can see that these points are well-defined for $\theta = 0, \pi$, and are continuous there. 
\VK{$\hat{L}_\pm (\iota(\z'))$} is the complex line in $\hat \S^2_{\c}$ passing through the points \VK{$\iota(\z')$} and
$[\bdeta_\pm (\z')]$. 


For the parametrisation \eqref{e:bdd001} for the point of $\Xi$, the formula \eqref{e:ap004}
reads as follows:
\begin{equation}
\beta = \beta_{\pm} (\theta, \vph) =  \arccos \left(    \frac{- i \cos \theta \, \cos \vph \mp \sin \vph}{\sin \theta} \right)
= 
\mp i \log ( i e^{\pm i \vph} \tan (\theta / 2)).
\label{e:ap006}
\end{equation}
If one takes real $\theta$ and $\vph$ in \eqref{e:ap006}, it yields two maps $\S^2 \to \Xi$ that coincide with the  
maps $\Phi_\pm$ defined in \eqref{e:map_Phi}:
\begin{equation}
\Phi_\pm (\x) = [\bdeta (\beta_\pm (\theta, \vph))], 
\qquad 
\x = (\theta, \vph) \in \S^2.
\label{e:ap006a}
\end{equation}

Let us take some $\bdeta(\beta)$ and solve \eqref{e:ap001}, \eqref{e:ap003} for~$\z'$.
After some algebra, we get 
\begin{equation}
\z' = \x_\pm (\bdeta)  + \alpha \bdeta,
\qquad 
\alpha \in \mathbb{C},
\label{e:ap007}
\end{equation}
where $
\x_{\pm} (\bdeta(\beta)) = 
(\theta_\pm (\beta), \vph_\pm (\beta) ),
$ with
\begin{alignat}{2}
	&	\theta_+ (\beta) =2 \arctan(|e^{i \beta}|), \quad  &&\varphi_+ (\beta)= {\rm Re}[\beta] 
    - \frac{\pi}{2}, \label{e:0229} \\
	&	\theta_- (\beta)=\pi-2 \arctan(|e^{i \beta}|), \quad &&\varphi_- (\beta)={\rm Re}[\beta] 
    + \frac{\pi}{2}.  \label{e:0230} 
\end{alignat}
The formulas \eqref{e:0229}, \eqref{e:0230} provide the properties b) and~d). 
Note that the points $\x_\pm (\bdeta)$ are chosen on the real sphere, so 
\eqref{e:0229}, \eqref{e:0230} provide two maps $\Xi \to \S^2$.
A careful check shows that these maps are smooth and coincide with the maps $\Psi_\pm$ defined in \eqref{e:map_Psi}:
\begin{equation}
\Psi_\pm ([\bdeta(\beta)])  = 
\x_{\pm} (\bdeta(\beta)).
\label{e:ap007}
\end{equation}
A direct check shows that both $\Psi_+$ and $\Psi_-$ are bijections between $\Xi$ and $\S^2$, and this provides 
property~c).

It remains to address the property~e).  Consider $\z''\in\S^2_{\c}$ and $\z\in\S^2_{\c}\cap\{\z\cdot\z''=1\}$. 
One can directly show that 
\[
\z'' \cdot (\z - \z'')  = 0, 
\qquad 
(\z - \z'') \cdot (\z - \z'') = 0. 
\]
So upon choosing $\bdeta$ as $\bz-\bz''$ we have $\bdeta\in\mathcal{N}$ and $\z=\z''+\bdeta$ is on a characteristic line passing through $\z''$, that is $\z\in L_+(\z'')\cup L_-(\z'')$. Moreover, if we assume that $\z\in L_+(\z'')\cup L_-(\z'')$, we know that there exists $\bdeta\in\mathcal{N}$ such that $\z-\z''=\bdeta$ and $\z''\cdot \bdeta=0$. This can be used directly to show that $\z\cdot\z''=1$. This the proof.

\sout{Finally, note we can extend the maps $\Phi_\pm$ and $\Psi_{\pm}$ 
to all points $\z \in \hat \S^2_{\c}$.
First, note that for $p\in \Xi$, we have $\Phi_{+}(p)=\Phi_-(p)=p$. Similarly, $\Psi_+(\x)=\Psi_-(\x)=\x$ for $\x \in \S^2$.
Since we already defined $\Phi_\pm (\S^2)$ and $\Psi_\pm (\Xi)$, we have only to define all four maps 
on $\S^2_{\c} \setminus \S^2$.

For $\x \in \S^2$, the maps $\Phi_\pm (\x)$ are given in the coordinate form by (\ref{e:ap006}). 
Note that the same formulae are valid for a complex point $\z$ defined by {\em complex\/} 
coordinates $(\theta , \varphi)$.  Thus, we define $\Phi_{\pm} (\z)$ by (\ref{e:ap006}). 

Finally, we define $\Psi_{\pm} (\z)$ by }
\sout{where the maps $\Psi_\pm$ in the right-hand sides in the coordinate form are defined by
(\ref{e:0229}), (\ref{e:0230}).}



\section{The CLBO in different coordinate systems}
\label{app:usefulformulas}
This appendix is dedicated to record the expressions of the CLBO in two different coordinate systems, which can be done through a standard change of variables procedure. 

We start with the affine complex coordinates $(z_1,z_2)$ on $\S^2_{\c}$, with the parametrisation
\begin{equation}
\z = \left( z_1, z_2, \pm \sqrt{1 - z_1^2 - z_2^2} \right).
\label{e:f004}
\end{equation}
for which we have
\begin{equation}
\tilde \Delta  = \frac{\ptl^2}{\ptl z_1^2} + \frac{\ptl^2}{\ptl z_2^2}
- \left( 
z_1^2 \frac{\ptl^2 }{\ptl z_1^2} + z_2^2 \frac{\ptl^2 }{\ptl z_2^2} + 
2 z_1 z_2 \frac{\ptl^2 }{\ptl z_1 \ptl z_2}
\right)  - 
2 \left( 
z_1 \frac{\ptl }{\ptl z_1} + z_2 \frac{\ptl }{\ptl z_2} 
\right) ,
\label{e:f008}
\end{equation}
In addition, for the coordinates \eqref{e:bdd002} used to approach $\Xi$, we have
\begin{equation}
\tilde \Delta  = (\epsilon^4 - \epsilon^2) \frac{\ptl^2 }{\ptl \epsilon^2}
+ \epsilon^3 \frac{\ptl}{\ptl \epsilon} + \epsilon^2 \frac{\ptl^2 }{\ptl \beta^2},
\label{e:f009}
\end{equation}
showing that the CLBO  is singular on~$\Xi$ (i.e.\ $\epsilon=0$), so it will only be considered on~$\S^2_{\c}$.

\section{Motivation for sliding contours: the planar Green's function}\label{app:planar}

Consider the real plane  $\mathbb{R}^2$ with coordinates $(y_1,y_2)$, bearing the Helmholtz equation 
with a point source located at the origin:
\begin{equation}
\left( 
\frac{\ptl^2}{\ptl y_1^2} +
\frac{\ptl^2}{\ptl y_2^2}
+ k^2
\right) u(y_1, y_2) = \delta(y_1) \delta(y_2).
\label{e:xdd010}
\end{equation}
Assume further that the field $u$ obeys the radiation condition (no waves should be coming from infinity). 
The detailed formulation of this condition is discussed in \cite{Babich2007} for example. Here, we assume the $e^{-i \omega t}$ convention, so we just need to  specify that the solution should behave as $e^{i k r}/\sqrt{kr}$, as $r = \sqrt{y_1^2 + y_2^2}$ tends to infinity. As is well known, $u$ can be written in terms of the Hankel function of the first kind as  $u(y_1,y_2) = -\frac{i}{4} H^{(1)}_0\left(k \sqrt{y_1^2+y_2^2}\right)$. Below, we show that it can also be written as a plane-wave representation using the sliding-contours approach.

Introduce four domains in the $(y_1,y_2)$ plane, illustrated in \figurename~\ref{f:shanin_07}, left and given by: 
\[
X^{(1)} : \, y_1 > 0, 
\qquad 
X^{(2)} : \, y_2 > 0,
\qquad
X^{(3)} : \, y_1 < 0, 
\qquad 
X^{(4)} : \, y_2 < 0.
\]

\begin{figure}[h]
    \centering{\includegraphics[width=0.9\textwidth]{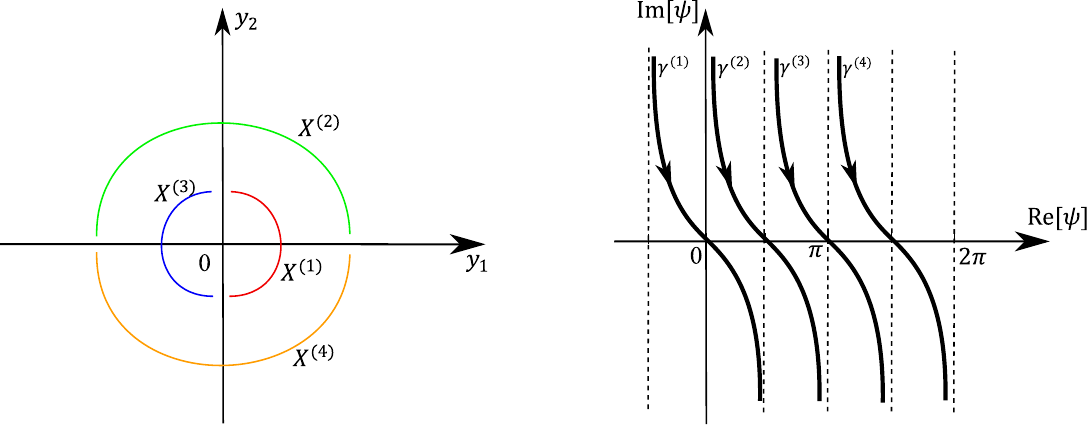}}
	\caption{Domains $X^{(m)}$ (left), and associated contours $\gamma^{(m)}$ (right)}
	\label{f:shanin_07}
\end{figure}

Using the Fourier transform and residual integration (see e.g.\ \cite[eq (B.7)]{Nethercote2020ReviewPaper}
, one can easily find the plane-wave 
representation for the field $u$: 
\begin{equation}
u(y_1 , y_2) = \frac{1}{4\pi i} \int_{\gamma^{(m)}} w(y_1, y_2 ; \psi) \, d\psi,
\qquad 
(y_1 , y_2) \in X^{(m)},
\label{e:xdd011}
\end{equation}
where 
\[
w(y_1, y_2 ; \psi) = \exp\{ i k (y_1 \cos \psi + y_2 \sin \psi) \}.
\]

Let us analyse the representation \eqref{e:xdd011}. 
The function $w$ is indeed a plane wave for this problem. The plane waves are 
indexed by a complex angular parameter $\psi$. The domain of $\psi$ is the complex plane, factorised
by the relation $\psi = \psi + 2\pi$, i.e. it is a complex circle~$\S^1_{\c}$. This set is the dispersion diagram and it is a complex manifold. 

The integrals \eqref{e:xdd011} can naturally be written in a form similar to \eqref{e:0436a}: 
\[
u (y_1, y_2) = \int_\gamma A(\psi) \, w(y_1, y_2 ; \psi) \, \Omega, 
\]
where 
$
A(\psi) = \frac{1}{4\pi i}
$
is a holomorphic spectral function on the dispersion diagram, 
and $\Omega = d\psi$ is a holomorphic 1-form on the dispersion diagram.  

One can see that the field $u(y_1 , y_2)$ cannot be described by a single integral everywhere on the plane 
(maybe minus the origin, where the source is), thus we have to cover the plane by several domains
$X^{(m)}$ and introduce the contours $\gamma^{(m)}$ for these domains. This constitutes the system of sliding contours
obeying the properties 1) -- 3) of Section~\ref{sec:slidingcontours}. 

Though not explicitly named in that way, the idea of sliding contours is discussed in the context of diffraction by wedges in for example \cite[Section 8.2]{Babich2007} and \cite[Appendix B]{Nethercote2020ReviewPaper}.

\section{Plane-wave representation for arbitrary source location, reciprocity and $\lambda$-symmetry.}
\label{app:sym}

Consider the Green's function $G(\x;\x_0)$ for an observation point $\x\in \S^2$ and a point source at $\x_0 \in \S^2$, with $\x\neq\x_0$. Take the reference point $\x^{\text{ref}}$ to be the antipode of the source $\x_0$, that is $\x^{\text{ref}}=-\x_0$. Given the rotational invariance of our problem, the formulae \eqref{e:0436}, \eqref{e:02324} and \eqref{e:cdd005} can be generalised to this case to get: 
\begin{equation}
G(\x ; \x_0)  = 
\mathcal{A}_\lambda
\int_{\gamma^{(m)}(\x_0)} w_2(\x,p;-\x_0) \Omega_{\x_0}, \quad \x\in X^{(m)}(\x_0),
\label{e:ddd001}
\end{equation}
where $\gamma^{(m)}(\x_0)$ and $X^{(m)}(\x_0)$ are sequences of contours and domains that can be obtained by simple rotation of the $\gamma^{(m)}$ and $X^{(m)}$ used in \eqref{e:cdd005}, and $w_2$ is given in Definition~\ref{def:planewaves} to be $w_2(\x,p;-\x_0)=\left(i \frac{\x \cdot \bdeta}{\x_0 \cdot \bdeta}\right)^{\lambda}$ if $p=[\bdeta]$. The constant $\mathcal{A}_\lambda$ is naturally defined by ${\mathcal{A}_\lambda=\frac{(-i)^{-\lambda}}{8 \pi \sin (\pi \lambda)}}$.
The form $\Omega_{\x_0}$ is an integration form linked to the point source $\x_0$
in the following way: 
it has simple poles at $\Phi_+ (\x_0)$ and $\Phi_- (\x_0)$, the residue at $\Phi_+ (\x_0)$
is $-i$, and the residue at $\Phi_- (\x_0)$ is~$i$. 
The contours $\gamma^{(m)}(\x_0)$ separate the points $\Phi_+ (\x)$ and $\Phi_-(\x_0)$ from the points $\Phi_-(\x)$ and $\Phi_+ (\x_0)$.
The points $\Phi_-(\x)$ and $\Phi_+ (\x_0)$ are located to the right of $\gamma^{(m)}(\x_0)$.
The formula \eqref{e:ddd001} is more universal than \eqref{e:cdd005} since it admits arbitrary 
positions of the source and observation point.  
   \VK{ \begin{remark}\label{rem:revVKE1}
        The fact that a meromorphic one-form on $\Xi$ with only simple poles is fully determined in terms of its residues follows from the theory of ``Abelian Differentials of the first kind'', as explained in \cite[Chapter 10]{MR0092855}.
    \end{remark}}
As is well-known, and clear from \eqref{e:02310},
the Green's function is reciprocal, i.e.\ :  
\begin{equation}
G(\x ; \x_0) = G(\x_0 ; \x).
\label{e:ddd006}
\end{equation}
This property is however not self-evident when $G$ is expressed using \eqref{e:ddd001}. The second part of this appendix is dedicated to show that this property can be recovered from \eqref{e:ddd001}. Due to rotational invariance, it is enough to show this for the specific case considered in the bulk of the article, that is $\x_0=\xNP$ (for which we use $\x^{\text{ref}}=-\xNP=\xSP$).

Let us consider the 1-forms $\omega_1$ and $\omega_2$ given by
\[
\omega_1 =  w_2(\x;p;\xSP) \Omega_{\xNP}, 
\qquad 
\omega_2 =  w_2(\xNP;p;-\x) \Omega_{\x},
\]
allowing us to write
\begin{alignat*}{2}
    G(\x;\xNP) &= \mathcal{A}_{\lambda} \int_{\gamma^{(m)}(\xNP)} \omega_1 \quad &&\text{ if } \x \in X^{(m)}(\xNP),\\
    G(\xNP;\x) &= \mathcal{A}_{\lambda} \int_{\gamma^{(m)}(\x)} \omega_2 \quad &&\text{ if } \xNP \in X^{(m)}(\x).
\end{alignat*}
For the form $\omega_1$, as explained in the bulk of the article we can use the parametrisation $p=[\bdeta(\beta)]$ and $\Omega_{\xNP}=d\beta$, where $\bdeta(\beta)$ is given in \eqref{e:bdd001}. Let us now consider the variable $\tau\equiv\tau_-=e^{-i\beta}$ on $\Xi$, and denote by $\tau^{\pm}(\x)$, the value of $\tau$ corresponding the the points $\Phi_{\pm}(\x)$ on $\Xi$. Introduce a change of variable $\tau \to t$ on $\Xi$ defined by
\begin{equation}
t(\tau)  = \tau^+ (\x) \frac{\tau - \tau^- (\x)}{\tau - \tau^+ (\x)}, \ \text{ with inverse } \  
\tau(t) = \tau^+(\x) \frac{t-\tau^-(\x)}{t-\tau^+(\x)}.
\label{e:ddd004}
\end{equation}
This naturally defines the map
\[
\Upsilon: \Xi \to \Xi, 
\qquad 
\tau \stackrel{\Upsilon}{\longmapsto} t (\tau). 
\]
After some algebra, using \eqref{e:ap006}, it is possible to show that 
\begin{equation}
\gamma^{(m)}(\x) = \Upsilon(\gamma^{(m)}(\xNP)), \quad \text{and} \quad \omega_1 = \Upsilon^* \omega_2.
\label{e:ddd005}
\end{equation}
The former implies that if $\x\in X^{(m)}(\xNP)$ then $\xNP\in X^{(m)}(\x)$, and the latter means that $\omega_1$ is a \textit{pullback} of $\omega_2$ under $\Upsilon$. We can therefore show that
\begin{align}
G(\x;\xNP)&=\mathcal{A}_{\lambda} \int_{\gamma^{(m)}(\xNP)} \omega_1=\mathcal{A}_{\lambda} \int_{\gamma^{(m)}(\xNP)} \Upsilon^*\omega_2 \nonumber \\ &=\mathcal{A}_{\lambda} \int_{\Upsilon(\gamma^{(m)}(\xNP))} \omega_2 = \mathcal{A}_{\lambda} \int_{\gamma^{(m)}(\x)} \omega_2=G(\xNP;\x),
\label{eq:procedure_reciprocity}
\end{align}
as required, where we have used \eqref{e:ddd005} and standard results regarding integration and change of variables, see e.g.\ \cite[Proposition 16.6, (d)]{Lee2012IntroductionSmoothManifolds}.

Note further that, using the well-known identity $P_\lambda=P_{-1-\lambda}$ \cite{Bateman1953}, and following the same procedure as the one used to prove the Lemma \ref{Lem:0441}, we would obtain
\begin{align}
	G(\x;\xNP) = \frac{(-i)^{1+\lambda}}{8 \pi \sin(\pi (-1-\lambda))}  \int_{\gamma}  w_1 (\x, [\bdeta(\beta)] ; \xSP ) 
     \, d\beta.
     \label{eq:raphE5}
\end{align}
Of interest to this appendix is the fact that through a procedure similar to that followed in \eqref{eq:procedure_reciprocity}, we could have obtained \eqref{eq:raphE5} directly from \eqref{e:cdd005}, and, doing so, we would actually recover the property $P_\lambda=P_{-1-\lambda}$ from the plane-wave representation.

\VK{Finally, let us give two more generalisations of the plane-wave representation \eqref{e:ddd001} where the locations of $\x, \x_0$ and $\x^{\mathrm{ref}}$ are all independent. 
\begin{proposition}\label{prop.D2}
  Let $\x$, $\x_0$ and $\x^{\text{ref}}$ be points on $\S^2$ such that $\x\neq\x_0$ and $\x\neq\x^{\text{ref}}$, and define $\mathcal{B}_{\lambda}=\tfrac{-i}{8 \pi \sin(\pi \lambda)} = \mathcal{B}_{-\lambda-1}$. Then, the following integral representation formulae are valid
    \begin{align}
    G (\boldsymbol{x} ; \x_0) &=\mathcal{B}_{\lambda}
    \int_{\gamma} w_2  (\boldsymbol{x}, p ; -
    \boldsymbol{x}^{\mathrm{ref}}) w_1 (\x_0, p ;
    \boldsymbol{x}^{\mathrm{ref}}) \Omega_{\boldsymbol{x}^{\mathrm{ref}}},
    \label{e:New01} \\
     G (\boldsymbol{x} ; \x_0) &= \mathcal{B}_{\lambda}
    \int_{\gamma} w_1 (\boldsymbol{x}, p ; -
    \boldsymbol{x}^{\mathrm{ref}}) w_2 (\x_0, p ;
    \boldsymbol{x}^{\mathrm{ref}}) \Omega_{\boldsymbol{x}^{\mathrm{ref}}}
    \label{e:New02},
    \end{align}
    where $\gamma$ is any contour of the form $\Phi(\Gamma)$ for some great circle $\Gamma$ on $\S^2$ separating the two points $\x_0$ and $\x^{\text{ref}}$ from the point $\x$. The contour $\gamma$ is understood to be oriented such that the points
  $\Phi_+ (\x_0)$ and $\Phi_- (\boldsymbol{x})$ lie to its left, and the points $\Phi_-
  (\x_0)$ and $\Phi_+ (\boldsymbol{x})$ lie to its  right.
  \label{lem:PWNew01}
\end{proposition}
Note that here, we use the notation $\gamma$ instead of $\gamma^{(m)}(\x_0)$ since the sliding contours procedure did not treat $\x_0$ and $\x^{\mathrm{ref}}$ as independent. However, due to Cauchy's theorem, the sliding contours procedure may be adapted verbatim to accommodate the formulae \eqref{e:New01} and \eqref{e:New02}. Indeed, it is the sliding contours procedure which justifies the definition of $\gamma=\Phi(\Gamma)$ purely in terms of how  the great-circle $\Gamma \subset \mathcal{S}^2$ separates the points $\x, \x_0$ and $\x^{\mathrm{ref}}$.  
}

\VK{In order to prove the proposition, we require the following Lemma.
\begin{lemma}
  \label{lem:New01} For any $p=[\bdeta]\in\Xi$, and any $\x_0$ and $\x^{\mathrm{ref}}$ in $\S^2$ such that $\bdeta\cdot\x^{\mathrm{ref}} \neq 0$, introduce the function
  \begin{equation}
      \mathscr{F}(\x_0,p;\x^{\mathrm{ref}})\equiv\left( \frac{\x_0 \cdot \boldsymbol{\eta}}{\boldsymbol{x
    }^{\mathrm{ref}} \cdot \boldsymbol{\eta}} \right).
  \end{equation}
  Then the following  equality is true:
  \begin{equation}
   \mathscr{F}(\x_0,p;\x^{\mathrm{ref}}) \Omega_{\x_0}(p) =
    \Omega_{\boldsymbol{x }^{\mathrm{ref}}}(p). \label{e:New09}
  \end{equation}
\end{lemma}
\begin{proof}
    By definition, $\Omega_{\x_0}$ is the unique meromorphic $1$-form on
  $\Xi$ which has exactly two poles, \ located at $\Phi_+ (\x_0)$ and
  $\Phi_- (\x_0)$ with residues at these poles equal to $- i$ and $+
  i$, respectively (cf. Remark \ref{rem:revVKE1}). The function $\mathscr{F}(\x_0,p;\x^{\mathrm{ref}})$ has exactly two simple zeros at $p=\Phi_+ (\x_0)$ and $p=\Phi_-
  (\x_0)$, and two simple poles at $p=\Phi_+ (\x^{\mathrm{ref}})$ and $p=\Phi_- (\tmmathbf{x }^{\mathrm{ref}})$. Therefore, the
  expression $\mathscr{F}(\x_0,p;\x^{\mathrm{ref}}) \Omega_{\x_0}$ is a
  meromorphic 1-form on $\Xi$ with poles located at $\Phi_+ (\tmmathbf{x
  }^{\mathrm{ref}})$ and $\Phi_- (\tmmathbf{x }^{\mathrm{ref}})$. Let us compute
  the residues at these poles. Using rotational invariance of characteristics (Lemma \ref{lem:charac}, f)), we may without loss of generality choose  
    \begin{align*}
    \x_0 = (0, 0, 1) ~ \text {and } ~ \x^{\mathrm{ref}} = (0, \sin (\alpha^{\mathrm{ref}}), \cos
    (\alpha^{\mathrm{ref}})),
    \end{align*}
  for some fixed $\alpha^{\mathrm{ref}} \in [0, 2 \pi)$. On $\Xi$, choose the
  $\tau -$coordinate system according to
  \begin{equation}
    \bdeta = ((\tau^2  + 1), i (1 - \tau^2), 2 i \tau)
  \end{equation}
  so for $p=[\bdeta]$, we have
  \begin{equation}
   \mathscr{F}(\x_0,p;\x^{\mathrm{ref}}) \Omega_{\x_0}(p) =
    \frac{- 2 i}{(1 - \tau^2) \sin (\alpha^{\mathrm{ref}}) + 2 \tau \cos
    (\alpha^{\mathrm{ref}})} d \tau .
  \end{equation}
  From here, it can be explicitly computed that for any fixed
  $\alpha^{\mathrm{ref}}$, the residue at the pole $\Phi_+ (\tmmathbf{x
  }^{\mathrm{ref}})$ equals $- i$ whereas the residue at the pole $\Phi_-
  (\tmmathbf{x }^{\mathrm{ref}})$ equals $+ i$. Thus, using the fact that a meromorphic 1-form on $\Xi$ is uniquely determined in terms of its residues, we find that $\mathscr{F}(\x_0,p,\x^{\mathrm{ref}}) \Omega_{\x_0}(p) = \Omega_{\tmmathbf{x
  }^{\mathrm{ref}}}(p)$, as required. 
\end{proof}
\begin{proof}[Proof of Proposition \ref{prop.D2}]
        Using Lemma \ref{lem:New01}, the proof of the proposition is reasonably straightforward. First, computations show that, for $p = [\bdeta]$, we have 
            \begin{align}
                w_2(\x,p; -\x_0) = e^{- i \pi (\lambda + 1)/2} w_2(\x,p;-\x^{\mathrm{ref}}) w_1(\x_0,p; \x^{\mathrm{ref}}) \mathscr{F}(\x_0,p;\x^{\mathrm{ref}}).
            \end{align}
         So, using Lemma \ref{lem:New01}, we may re-write \eqref{e:ddd001} as follows. 
         \begin{align*}
            G(\x;\x_0) &= \mathcal{A}_{\lambda} \int_{\gamma} w_2(\x,p; -\x_0) \Omega_{\x_0} \\
            &=  \mathcal{A}_{\lambda} \int_{\gamma} e^{- i \pi (\lambda + 1)/2} w_2(\x,p; - \x^{\mathrm{ref}}) w_1(\x_0,p;\x^{\mathrm{ref}}) \mathscr{F}(\x_0,p;\x^{\mathrm{ref}})\Omega_{\x_0} \\
            & = \mathcal{B}_{\lambda}
    \int_{\gamma} w_2  (\boldsymbol{x}, p ; -
    \boldsymbol{x}^{\mathrm{ref}}) w_1 (\x_0, p ;
    \boldsymbol{x}^{\mathrm{ref}}) \Omega_{\boldsymbol{x}^{\mathrm{ref}}}. \numberthis
         \end{align*}
       Using \eqref{eq:raphE5}, the equality \eqref{e:New02} is proved similarly but we omit the computational details for brevity. 
    \end{proof}}

\AS{
\section{Computation of contour integrals on a complex manifold}
\label{app:int}

 The integral representation  (\ref{e:ddd001}) is the main result of the paper. Such an integral 
may seem abstract, so here our aim is to demonstrate that it can be evaluated numerically 
using a quadrature formula. Note that the theoretical definition of such an integral 
combines the definition of an integral of a differential form over a chain on a manifold (see \cite{Collier} or any other introduction to calculus on smooth manifolds) with the usual 1D complex analysis (see \cite{Shabat1}).
Throughout this appendix the complex manifold will be assumed to be one-complex-dimensional as it coincides with the situation of the paper, but similar ideas can be applied to higher-dimensional manifolds.

Consider a contour integral 
\begin{equation}
I = \int_\gamma f(p) \Omega (p), 
\label{e:apE01}
\end{equation}
where $\gamma$ is an oriented contour drawn on a 1D complex manifold~$M$. The function
$f(p)$, $p \in M$, is a holomorphic function on $M$ (at least near $\gamma$), and $\Omega (p)$ is a holomorphic 
(at least near $\gamma$) differential 1-form on~$M$. 

Note that the expression 
$f \Omega$ is slightly redundant, since $f \Omega$ is itself a holomorphic 1-form, so it is not necessary to write the function~$f$. We keep it only because it looks more 
usual. 

Cover $M$ with a set of charts $W_j$. By definition, in each chart one can introduce 
a local complex variable~$t_j$. For any non-empty intersection $W_j \cap W_m$, 
the transition functions $t_j (t_m)$ and $t_m (t_j)$ are holomorphic. 

Introduce a dense enough mesh of successive points (nodes) $p^{(1)}\dots p^{(N)}$ 
on $\gamma$. We assume that 
$\gamma$ is approximated by a union of small segments $[p^{(n)} , p^{(n+1)}]$, and the orientation 
of the contour $\gamma$ is such that it goes from $p^{(n)}$ to~$p^{(n+1)}$.
It may happen that the last node $p^{(N)}$ coincides with 
the first node $p^{(1)}$.
Assume that for each segment $[p^{(n)} , p^{(n+1)}]$ 
one can select a chart $W_j$, $j = j(n)$ to which the segment fully belongs.

Consider a segment $[p^{(n)} , p^{(n+1)}]$ in the chart $W_j$, $j = j(n)$. 
In the coordinate $t_j$, the nodes $p^{(n)} , p^{(n+1)}$ have coordinates 
$t_j^{(n)}$, $t_j^{(n+1)}$. Besides, by the definition of a holomorphic function, 
$f(p)$ on the segment can be written as a function $p_n(t_j)$ holomorphic with respect to~$t_j$.
Moreover, again by definition, $\Omega(p)$ on this segment
can be written as 
\[
\Omega = h_n(t_j) \, d t_j,
\]
where $h_n (t_j)$ is a holomorphic function of the variable~$t_j$.

An approximation of $I$ obtained by the trapezoidal quadrature rule is as follows: 
\begin{equation}
I \approx \frac{1}{2}\sum_{n = 1}^{N-1} 
\left(
f_n(t^{(n+1)}_{j(n)}) h_n (t^{(n+1)}_{j(n)}) + f_n(t^{(n)}_{j(n)}) h_n (t^{(n)}_{j(n)})
\right)
\left( 
t^{(n+1)}_{j(n)} - t^{(n)}_{j(n)}
\right)
.
\label{e:apE02}
\end{equation}
This formula can be implemented numerically. 

Note that the formula (\ref{e:02324}) is simpler than (\ref{e:apE01}), since everywhere on the 
contour one can choose a single chart described with the variable $t = \beta$.
In this case, $\Omega = d\beta$ and $h \equiv 1$. Under these simplifications, (\ref{e:apE02}) can still be used for computations. 
}

\end{document}